\documentclass[11pt, a4paper]{amsart}
\usepackage[utf8]{inputenc}
\usepackage{enumerate}
\usepackage{amsmath}
\usepackage{amsthm}
\usepackage{amsfonts}
\usepackage{amssymb}
\usepackage[initials, nobysame]{amsrefs}
\usepackage{mathrsfs}
\usepackage{thm-restate}

\newtheorem{thm}{Theorem}[section]
\newtheorem{lem}[thm]{Lemma}
\newtheorem{prop}[thm]{Proposition}

\theoremstyle{definition}
\newtheorem*{defn}{Definition}
\newtheorem{example}[thm]{Example}

\theoremstyle{remark}
\newtheorem*{note}{Note}
\newtheorem{rmk}[thm]{Remark}

\newcommand{\Addresses}{{
  \bigskip
  \footnotesize

  \textsc{Graduate School of Mathematics, Brown University,
    Providence, RI 02912}\par\nopagebreak
  \textit{E-mail address}: \texttt{thomas\_silverman@brown.edu}
}}
\newcommand\restrict[1]{\raisebox{-.5ex}{$|$}_{#1}}

\begin{document}
\title{A non-archimedean $\lambda$-lemma}
\author{Thomas Silverman}

\begin{abstract}
We provide a framework for studying the dynamics of families of one-variable rational functions parametrized by Berkovich spaces over a complete non-archimedean field.  We prove a non-archimedean analogue of Ma\~{n}\'{e}, Sad, and Sullivan's $\lambda$-Lemma and use this to show an equivalence of two stability conditions for families of rational functions parametrized by an open subset of the Berkovich affine line.
\end{abstract}

\maketitle

\section{Introduction}
\subsection{Statement of Main Result}
In a celebrated paper \cite{MSS83}, R. Ma\~{n}\'{e}, P. Sad, and D. Sullivan prove a result about holomorphic families of injections called the $\lambda$-Lemma with impressive applications to the complex dynamics of families of one-variable rational functions.  The \textit{Julia set} of a rational function $f \in \mathbb{C}(z)$ is a compact set where the dynamics of $f$ are chaotic; its complement is called the \textit{Fatou set}.  Ma\~{n}\'{e}-Sad-Sullivan introduce a notion of \textit{$J$-stability} for a holomorphic family of rational functions and, using the $\lambda$-Lemma, prove that two characterizations of $J$-stability are equivalent; one characterization involves homeomorphisms between the Julia sets that commute with the dynamics, and the other involves periodic points maintaining the same local behavior.  See \cite{MSS83}*{Theorem B} for the precise statement, and see \cite{McM94}*{Theorem 4.2} for these and more equivalent characterizations of $J$-stability.  The main goal of this article is to prove a non-archimedean analogue of the $\lambda$-Lemma and use this to prove the following equivalence of two stability conditions in non-archimedean dynamics.

\begin{restatable}{thm}{mainthm} \label{thm:main}
Let $K$ be a complete algebraically closed non-archimedean field, let $U \subseteq \mathbb{A}^{1,\text{an}}_K$ be a connected open set, and let $\left\lbrace f_\lambda \right\rbrace_{\lambda \in U}$ be an analytic family of rational maps over $K$.  Assume that the type I repelling periodic points are in dense in the Julia set $J_{x_0}$ for some $x_0 \in U$.  Then, the following are equivalent:

\begin{enumerate}[$(1)$]
\item The Julia sets $J_\lambda$ move analytically on $U$.
\item For all $x \in U$, $f_x$ has no unstably indifferent periodic points and no type I repelling periodic points with multiplicity greater than $1$.
\end{enumerate}
\end{restatable}

We will postpone the precise definitions appearing in the statement until \S \ref{sec:deflem}, but a few preliminary remarks are in order.   When studying families in complex dynamics, an important idea is that the Julia set dynamics are the same for each rational function whose parameter belongs to a given connected component of the set of $J$-stable parameters.  We would like to employ this idea in the non-archimedean case as well, but $K$ equipped with its metric topology is totally disconnected.  The solution to this difficulty in the context of a single rational function over a non-archimedean field is to use the Berkovich projective line $\mathbb{P}^{1,\text{an}}_K$ as the dynamical space, which is a path-connected compact Hausdorff topological space containing $\mathbb{P}^{1}(K)$ as a topological subspace.  This strategy has achieved tremendous success in recent years beginning with J. Rivera-Letelier's thesis \cite{RL03} but has not yet been used for parameter spaces of families of rational functions.  In this article, we will use Berkovich analytic spaces for both dynamical and parameter spaces; the parameter space $U$ and the Julia sets $J_\lambda$ appearing in the statement of Theorem \ref{thm:main} are subsets of the Berkovich projective line.  One wrinkle that arises as a consequence of this approach is that the parameter values that are not classical points do not correspond to rational functions defined over $K$, so there is no immediately apparent way of associating dynamical systems to them.  However, as we will see in \S \ref{sec:deflem}, they can be naturally associated to rational functions defined over an extension of $K$.  Still, the fact that the dynamical space depends on the parameter value is a feature that does not appear in the complex case and will require us to be more careful.

We also mention that while the ``no unstably indifferent periodic points'' condition is a direct analogue of the periodic point characterization of $J$-stability in \cite{MSS83}, the ``no higher multiplicity repelling periodic points'' condition refers to a bifurcation that does not appear in complex dynamics.  In fact, this bifurcation can only occur at non-classical Berkovich space parameter values, which helps demonstrate the utility of this approach.  We will provide a more detailed discussion of these bifurcations in Example \ref{ex:quadratics}.

The topic of $J$-stability in non-archimedean dynamics has already been studied by J. Lee \cite{Lee17}, but in that article the author proves an analogue of a different result from \cite{MSS83} which states that hyperbolic parameters are $J$-stable.  Lee also does not use Berkovich spaces at all in his work, so our approaches differ considerably.

The non-archimedean analogue of the $\lambda$-Lemma we will need is the following.

\begin{restatable}{thm}{lamlemma}[The Non-archimedean $\lambda$-Lemma] \label{thm:lamlemma}
Let $U \subseteq \mathbb{A}^{1,\text{an}}$ be a connected open subset and let $x_0 \in U$ be given.  Let $E \subseteq \mathbb{P}^{1,\text{an}}_{x_0}$ be a set of type I points.  Let
\[
\phi_\lambda: E \longrightarrow \mathbb{P}^{1,\text{an}}_{\lambda}
\]
be a family of injections (not necessarily continuous) such that $\phi_{x_0} = \text{id}$ and $\lambda \mapsto \phi_\lambda(\xi)$ is an analytic map from $U$ to $\mathbb{P}^{1,\text{an}} \times U$ for each $\xi \in E$.  Then, each $\phi_\lambda$ has a unique continuous injective extension to the topological closure $\overline{E}$ of $E$ in $\mathbb{P}^{1, \text{an}}_{x_0}$.  Moreover, $\lambda \mapsto \phi_\lambda(\xi)$ is an analytic map from $U$ to $\mathbb{P}^{1,\text{an}} \times U$ for each type I point $\xi \in \overline{E}$.
\end{restatable}

It is important to note that this result requires a set of type I points $E$, which is why need the assumption that the Julia set contains type I repelling periodic points in Theorem \ref{thm:main}.  See \S \ref{sec:deflem} for further discussion.

\subsection{Contents of this Paper}
In \S \ref{sec:berkgeo}, we provide background information on Berkovich's theory of analytic geometry, including analogues of the open mapping theorem and the inverse function theorem from complex analysis.  We also provide specifics pertaining to the Berkovich projective line in \S \ref{ssec:bpl}.

In \S \ref{sec:dynbpl}, we recall basic definitions and facts about dynamics of rational functions on the Berkovich projective line.

In \S \ref{sec:deflem}, we provide the definitions appearing in the statement of Theorem \ref{thm:main} and prove some related lemmas.  We also discuss these definitions in the context of a specific example.

In \S \ref{ssec:lamlemma}, we state and prove the non-archimedean analogue of the $\lambda$-Lemma and use it to prove Theorem \ref{thm:main} in \S \ref{ssec:pfmain}.

\section{Berkovich Analytic Geometry Background} \label{sec:berkgeo}
\subsection{General Theory} \label{ssec:gentheory}
The material in this subsection can be found in \cite{Ber90} and \cite{Ber93}.  See also \cite{DFN15}*{Part I} for an overview of Berkovich analytic spaces.  Throughout, $K$ will denote an algebraically closed field complete with respect to a nontrivial non-archimedean absolute value $\left\lvert \cdot \right\rvert$.  Although our ground field will always be algebraically closed, we will also need to make use of some of theory for transcendental extensions of $K$ that are not algebraically closed.  We will use $\mathbb{K}$ to denote a field complete with respect to a nontrivial non-archimedean absolute value that is not necessarily algebraically closed.  We will write $\widetilde{\mathbb{K}}$ for the residue field
\[
\widetilde{\mathbb{K}} := \left\lbrace x \in \mathbb{K} \mid \left\lvert x \right\rvert \leq 1 \right\rbrace / \left\lbrace x \in \mathbb{K} \mid \left\lvert x \right\rvert < 1 \right\rbrace.
\]
of a non-archimedean field $\mathbb{K}$.

In analogy with affine spaces in the theory of schemes, the basic building blocks of Berkovich's analytic theory are the \textit{affinoid spaces} to be defined shortly; we will need a few preliminary definitions.  A \textit{Banach ring} is a complete normed ring (we will assume that all norms are non-archimedean).  A seminorm $\left\lvert \cdot \right\rvert_x$ on a Banach ring $(\mathscr{A}, \left\lVert \cdot \right\rVert)$ will be called \textit{bounded} if there is a constant $C \geq 0$ such that $\left\lvert f \right\rvert_x \leq C \left\lVert f \right\rVert$ for all $f \in \mathscr{A}$.  Given a Banach ring $\mathscr{A}$ and an $n$-tuple of positive real numbers $\textbf{r} = (r_1, \ldots ,r_n)$, we will equip the polynomial ring $\mathscr{A}[\textbf{T}] := \mathscr{A}[ T_1, \ldots, T_n ]$ with the norm
\[
\left\lVert \sum_{\nu \in \mathbb{Z}_{\geq 0}^n} a_\nu \textbf{T}^\nu \right\rVert_{\textbf{r}} := \max_\nu \left\lVert a_\nu \right\rVert \textbf{r}^\nu,
\]
where $\textbf{T}^\nu = T_1^{\nu_1} T_2^{\nu_2} \cdots T_n^{\nu_n}.$  The completion of $\mathscr{A}[ \textbf{T} ]$ with respect to this norm is the Banach $\mathscr{A}$-algebra denoted by $\mathscr{A} \left\lbrace \textbf{r}^{-1} \textbf{T} \right\rbrace$.  This ring can be viewed as the ring of convergent power series over $\mathscr{A}$ on the polydisc with polyradius $\textbf{r}$ centered at the origin.  A \textit{$\mathbb{K}$-affinoid algebra} is a Banach $\mathbb{K}$-algebra isomorphic to a quotient of $\mathbb{K} \left\lbrace \textbf{r}^{-1} \textbf{T} \right\rbrace$ equipped with the quotient norm induced by $\left\lVert \cdot \right\rVert_\textbf{r}$ for some $\textbf{r} \in \mathbb{R}_+ ^n$.  Given a $\mathbb{K}$-affinoid algebra $\mathscr{A}$, its associated \textit{affinoid spectrum} is the topological space
\[
\mathcal{M}(\mathscr{A}) := \left\lbrace \text{bounded multiplicative seminorms } \left\lvert \cdot \right\rvert_x  \text{on } \mathscr{A}  \right\rbrace
\]
equipped with the weakest topology such that $\left\lvert \cdot \right\rvert_x \mapsto \left\lvert f \right\rvert_x$ is continuous for all $f \in \mathscr{A}$.  Affinoid spectra are always compact Hausdorff topological spaces \cite{Ber90}*{Theorem 1.2.1}.  To emphasize that the seminorms are points in the space $\mathcal{M}(\mathscr{A})$, we will frequently use the notation $x \in \mathcal{M}(\mathscr{A})$ and $f \mapsto \left\lvert f(x) \right\rvert$ for the associated seminorm.  Any bounded homomorphism of Banach $\mathbb{K}$-algebras
\[
\phi: \mathscr{A} \longrightarrow \mathscr{B}
\]
induces a \textit{morphism of affinoid spectra}
\[
\phi^{\#}: \mathcal{M}(\mathscr{B}) \longrightarrow \mathcal{M}(\mathscr{A})
\]
defined by $\left\lvert f(\phi^{\#}(x)) \right\rvert = \left\lvert (\phi \circ f)(x) \right\rvert$ for $f \in \mathscr{A}$ and $x \in \mathcal{M}(\mathscr{B})$.
A closed subset $V \subseteq \mathcal{M}(\mathscr{A})$ is called a \textit{$\mathbb{K}$-affinoid domain} if there exists a $\mathbb{K}$-affinoid algebra $\mathscr{A}_V$ and a morphism of $\mathbb{K}$-affinoid spectra
\[
\phi: \mathcal{M}(\mathscr{A}_V) \longrightarrow \mathcal{M}(\mathscr{A})
\]
whose image is $V$ such that any morphism of $\mathbb{K}$-affinoid spectra $\mathcal{M}(\mathscr{B}) \longrightarrow \mathcal{M}(\mathscr{A})$ with image in $V$ factors uniquely through $\phi$.  It is clear that in this case the $\mathbb{K}$-affinoid algebra $\mathscr{A}_V$ is determined up to unique isomorphism, so we can identify $V$ with $\mathcal{M}(\mathscr{A}_V)$.  Now, we can provide a $\mathbb{K}$-affinoid spectrum $X = \mathcal{M}(\mathscr{A})$ with a sheaf of rings $\mathcal{O}_X$ defined as follows:
\[
\mathcal{O}_X(U) := \lim_\leftarrow \mathscr{A}_V,
\]
where the limit is taken over affinoid domains $V$ contained in the open set $U$.  The locally ringed space $(X, \mathcal{O}_X)$ is called a $\mathbb{K}$-affinoid space.

For our purposes, we will only need the Berkovich analytic spaces considered in \cite{Ber90}, which are called ``good'' analytic spaces in \cite{Ber93}.  The more general construction given in the latter paper is a more satisfying definition and includes many natural spaces that are not ``good'' including fibered products of good spaces, but it is also more technical.  We will opt for a brief presentation of the less technical and more familiar looking definition here.  A \textit{$\mathbb{K}$-quasiaffinoid} space is a pair $(\mathcal{U},\phi)$ with $\mathcal{U}$ a locally ringed space and $\phi: \mathcal{U} \longrightarrow X$ an open immersion into a $\mathbb{K}$-affinoid space $X$.  A compact subset $U \subseteq \mathcal{U}$ is called an \textit{affinoid domain} if $\phi(U)$ is an affinoid domain in $X$.  A \textit{morphism of $\mathbb{K}$-quasiaffinoid spaces}
\[
(\mathcal{U},\phi) \longrightarrow (\mathcal{V},\psi)
\]
is a morphism of locally ring spaces $\theta: \mathcal{U} \longrightarrow \mathcal{V}$ such that for any pair of afffinoid domains $U \subseteq \mathcal{U}$ and $V \subseteq \mathcal{V}$ with $\theta(U)$ contained in the topological interior of $V$ in $\mathcal{V}$, the induced map on affinoid algebras is a bounded homomorphism.  A \textit{$\mathbb{K}$-analytic atlas} on a locally ringed space $X$ is a collection $\left\lbrace (\mathcal{U}_i, \phi_i) \right\rbrace_{i \in I}$ of $\mathbb{K}$-quasiaffinoid spaces (called \textit{charts}) such that
\begin{itemize}
\item The collection $\left\lbrace \mathcal{U}_i \right\rbrace_{i \in I}$ is an open cover of $X$.
\item The induced map
\[
\phi_j \circ \phi_i^{-1}: \phi_i(\mathcal{U}_i \cap \mathcal{U}_j) \longrightarrow \phi_j(\mathcal{U}_i \cap \mathcal{U}_j)
\]
is an isomorphism of $\mathbb{K}$-quasiaffinoid spaces for each pair $i,j \in I$.
\end{itemize}
A $\mathbb{K}$-quasiaffinoid space $(\mathcal{U}, \phi)$ with $\mathcal{U} \subseteq X$ open is said to be \textit{compatible} with the atlas $\left\lbrace (\mathcal{U}_i, \phi_i) \right\rbrace_{i \in I}$ if each $\phi \circ \phi_i^{-1}$ is an isomorphism of $\mathbb{K}$-quasiaffinoid spaces as above.  Two atlases are \textit{compatible} if each open set in one atlas is compatible with the other.  A (good) \textit{Berkovich $\mathbb{K}$-analytic space} is a locally ringed space $X$ with an equivalence class of $\mathbb{K}$-analytic atlases.  A \textit{morphism} $\theta: X \longrightarrow Y$ between such spaces is a morphism of locally ringed spaces such that there are atlases $\left\lbrace (\mathcal{U}_i,\phi_i) \right\rbrace$ and $\left\lbrace (\mathcal{V}_j, \psi_j) \right\rbrace$ on $X$ and $Y$ respectively with
\[
\psi_j \circ \theta \circ \phi_i^{-1}: \phi_i(\mathcal{U}_i) \longrightarrow \psi_j(\mathcal{V}_j)
\]
a morphism of $\mathbb{K}$-quasiaffinoid spaces for all $i,j$.  To distinguish from maps between Berkovich spaces that are not morphisms that will arise below, we will also refer to morphisms as \textit{analytic maps}.

Affinoid domains in an analytic space can be defined by an analogous universal property as in the case of affinoid spaces: a compact subset $V$ of a $\mathbb{K}$-analytic space $X$ is called a \textit{$\mathbb{K}$-affinoid domain} if there is an analytic map $\mathcal{M}(\mathscr{A}) \longrightarrow X$ from an affinoid space that maps homeomorphically onto $V$ that any other analytic map with image in $V$ factors through uniquely.  Affinoid neighborhoods of a point $x \in X$ form a basis of neighborhoods \footnote{In fact, this property characterizes ``good'' spaces in the context of Berkovich's more general definition.}.  Given a point $x \in X$, we will write $\kappa(x)$ for the residue field $\mathcal{O}_{X,x} / \mathfrak{m}_{X,x}$.  Note that $\mathfrak{m}_{X,x}$ may be trivial.  The requirement that the norms on the affinoid algebras corresponding to affinoid domains in affinoid spaces are quotient norms implies that the ring $\mathcal{O}_{X,x}$ inherits a norm, which then gives an absolute value $\left\lvert \cdot \right\rvert_x$ on $\kappa(x)$.  We define $\mathcal{H}(x)$ to be the completion of $\kappa(x)$ with respect to this absolute value.  Given any affinoid neighborhood $\mathcal{M}(\mathscr{A})$ of $x$ in $X$, we also have that $\mathcal{H}(x)$ is the completion of the field $\mathscr{A}_{\mathfrak{p}_x} / \mathfrak{p}_x \mathscr{A}_{\mathfrak{p}_x}$
with respect to the absolute value $f \mapsto \left\lvert f(x) \right\rvert$, where $\mathfrak{p}_x = \left\lbrace f \in \mathscr{A} \mid \left\lvert f(x) \right\rvert = 0 \right\rbrace$.  We caution that $\mathcal{O}_X$ and $\kappa(x)$ are only useful objects to consider for good Berkovich spaces, but there is a compatible definition of $\mathcal{H}(x)$ that is useful for the more general definition of Berkovich spaces from \cite{Ber93}.

The field $\kappa(x)$ also has the important property that it is \textit{quasicomplete} \cite{Ber93}*{Theorem 2.3.3}, meaning that its absolute value extends uniquely to any algebraic field extension.  We will need the following fact about quasicomplete fields (but only for the proof of Lemma \ref{lem:ppcol}).

\begin{prop} \label{prop:weaklystab}
Let $\kappa$ be a quasicomplete field.  Then,
\[
[ \widehat{L} : \widehat{\kappa} ] = [ L : \kappa ]
\]
for every finite separable field extension $L$ of $\kappa$.
\end{prop}

Here $\widehat{L}$ denotes the completion of $L$ with respect to the unique extension absolute value on $L$ guaranteed by the quasicompleteness of $\kappa$.

\begin{proof}
This follows by combining Proposition 3.5.1/3, Corollary 3.2.3/2, and Proposition 2.3.3/6 in \cite{BGR84}.
\end{proof}

A morphism of Berkovich spaces $\phi: Y \longrightarrow X$ is said to be $\textit{finite}$ if for every $x \in X$ there is an affinoid neighborhood $\mathcal{M}(\mathscr{A})$ of $x$ in $X$ such that $\phi^{-1}\left(\mathcal{M}(\mathscr{A}) \right) = \mathcal{M}(\mathscr{B})$ is an affinoid domain in $Y$ and $\mathscr{B}$ is a finite $\mathscr{A}$-algebra.  The \textit{multiplicity} of a finite morphism of Berkovich spaces $\phi:Y \longrightarrow X$ at a point $y \in Y$ is the positive integer
\[
m_\phi(y) := \text{dim}_{\kappa(x)} \left( \mathcal{O}_{Y,y} / \mathfrak{m}_{X,x} \mathcal{O}_{Y,y} \right),
\]
where $x := \phi(y)$.  This definition of multiplicity was first used in a dynamical context in A. Thuiller's thesis \cite{Thu05} and mirrors the corresponding notion in algebraic geometry.

Now, we will discuss the specific Berkovich spaces that will arise in our work.  The \textit{closed $n$-dimensional Berkovich polydisc over $\mathbb{K}$} of polyradius $\textbf{r} = (r_1, \ldots, r_n) \in \mathbb{R}_+^n$ and center $\textbf{a} = (a_1, \ldots, a_n) \in \mathbb{K}^n$ is the affinoid space
\[
\overline{D}(\textbf{a}, \textbf{r}) := \mathcal{M} \left( \mathbb{K} \left\lbrace r_1^{-1} (T_1-a_1), \ldots , r_n^{-1} (T_n-a_n) \right\rbrace \right).
\]
Note that if $r_i \leq s_i$ for all $i$, the space $\overline{D}(\textbf{a}, \textbf{r})$ is an affinoid domain in $\overline{D}(\textbf{a}, \textbf{s})$.  The \textit{open $n$-dimensional Berkovich polydisc over $\mathbb{K}$} of polyradius $\textbf{r}$ and center $\textbf{a}$ is the quasiaffinoid space given by the open subset
\[
D(\textbf{a},\textbf{r}) := \left\lbrace x \in \overline{D}(\textbf{a}, \textbf{r}) \mid \left\lvert T_i -a_i \right\rvert_x < r_i \text{ for all } i \right\rbrace.
\]
One small pitfall with this notation is that the closed polydisc $\overline{D}(\textbf{a}, \textbf{r})$ is not the topological closure of the open polydisc $D(\textbf{a}, \textbf{r})$.  The \textit{$n$-dimensional Berkovich affine space over $\mathbb{K}$} is the set $\mathbb{A}^{n,\text{an}}_\mathbb{K}$ of all multiplicative seminorms on $\mathbb{K}[\textbf{T}]=\mathbb{K}[T_1, \ldots, T_n]$ whose restriction to $\mathbb{K}$ is bounded \footnote{This condition actually implies that its restriction to $\mathbb{K}$ agrees with the absolute value on $\mathbb{K}$.}, equipped with the weakest topology such that $x \mapsto \left\lvert f(x) \right\rvert$ is continuous for all $f \in \mathbb{K}[\textbf{T}]$.  Note that the open (and closed) $n$-dimensional Berkovich polydiscs can be naturally identified as subsets of $\mathbb{A}^{n,\text{an}}_\mathbb{K}$, which endows them with an analytic structure.  The structure sheaf $\mathcal{O}$ is the set of locally uniform limits of $n$-variable rational functions without poles in an appropriate sense.  Note that this implies that uniform limits of analytic functions on compact subsets of Berkovich affine space are analytic.   Also, for any Berkovich $\mathbb{K}$-analytic space $X$, there is a natural correspondence between global sections of the structure sheaf $\mathcal{O}_X$ and the set of morphisms from $X$ to $\mathbb{A}^{1,\text{an}}_\mathbb{K}$.

As with the corresponding variety, the \textit{$n$-dimensional Berkovich projective space} $\mathbb{P}^{n,\text{an}}_\mathbb{K}$ can be constructed in several ways:
\begin{itemize}
\item Using homogeneous coordinates: A multiplicative seminorm $x$ on $\mathbb{K}[T_0, \ldots, T_n]$ is called \textit{homogeneous} if it is determined by its values on homogeneous elements $f_d$ via the formula
\[
\left\lvert \sum_d f_d(x) \right\rvert = \max_d \left\lvert f_d(x) \right\rvert.
\]
We can view $\mathbb{P}^{n,\text{an}}_\mathbb{K}$ as the set of equivalence classes of homogeneous multiplicative seminorms on $\mathbb{K}[T_0, \ldots, T_n]$ that do not vanish identically on the set of degree one homogeneous polynomials and whose restriction to $\mathbb{K}$ is bounded, under the equivalence relation defined as follows: $x \sim y$ if and only there exists a constant $C$ such that $\left\lvert f_d(x) \right\vert = C^d \left\lvert f_d(y) \right\vert$ for all $d$ and all degree $d$ homogeneous polynomials $f_d$.
\item We can also construct $\mathbb{P}^{n,\text{an}}_\mathbb{K}$ by gluing $n+1$ copies of $\mathbb{A}^{n,\text{an}}_\mathbb{K}$ or by gluing $n+1$ copies of the $n$-dimensional closed unit polydisc $\overline{D}(\textbf{0},\textbf{1})$.  In both cases, the gluing is obtained by using $T_0/T_i, \ldots, \widehat{T_i/T_i}, \ldots T_n/T_i$ for the coordinates on the $i$-th copy.  This construction provides an analytic atlas that gives $\mathbb{P}^{n,\text{an}}_\mathbb{K}$ the structure of a Berkovich space.
\item For the one-dimensional case, we can view the Berkovich projective line as the one point compactification of the Berkovich affine line $\mathbb{P}^{1,\text{an}}_\mathbb{K} = \mathbb{A}^{1,\text{an}}_\mathbb{K} \cup \left\lbrace \infty \right\rbrace$.
\end{itemize}
For both affine and projective Berkovich spaces, we will omit the subscript when working over the default algebraically closed non-archimedean field $K$, i.e. $\mathbb{A}^{n,\text{an}} := \mathbb{A}^{n,\text{an}}_K$ and $\mathbb{P}^{n,\text{an}} := \mathbb{P}^{n,\text{an}}_K$.

The most important Berkovich space for our purposes is the product space $\mathbb{P}^{1,\text{an}} \times \mathbb{A}^{1,\text{an}}$.  Note that, as in the case of schemes, this is not a topological product space.  Instead, we can construct this space by gluing two copies of $\mathbb{A}^{2,\text{an}}$; viewing $\mathbb{A}^{2,\text{an}}$ as the set of multiplicative seminorms on the algebra $K[z,\lambda]$ with $z$ as the dynamical space coordinate and $\lambda$ as the parameter space coordinate, the gluing is given by $z \mapsto 1/z$.  We can also view $\mathbb{P}^{1,\text{an}} \times \mathbb{A}^{1,\text{an}}$ as a set of equivalence classes of homogeneous multiplicative seminorms on $K[\lambda][Z_0,Z_1]$ as above.  This product comes equipped with canonical analytic projections $\pi_1: \mathbb{P}^{1,\text{an}} \times \mathbb{A}^{1,\text{an}} \longrightarrow \mathbb{P}^{1,\text{an}}$ and
$\pi_2: \mathbb{P}^{1,\text{an}} \times \mathbb{A}^{1,\text{an}} \longrightarrow \mathbb{A}^{1,\text{an}}$ induced by the inclusions $K[Z_0,Z_1] \hookrightarrow K[Z_0,Z_1,\lambda]$ and $K[\lambda] \hookrightarrow K[Z_0,Z_1,\lambda]$.  Given an open set $U \subseteq \mathbb{A}^{1,\text{an}}$, we obtain a corresponding open set $\mathbb{P}^{1,\text{an}} \times U := \pi_2^{-1}(U)$, which has an induced Berkovich space structure.  Given a point $x \in \mathbb{A}^{1,\text{an}}$, the fiber $\pi_2^{-1}(x)$ can be naturally identified with the Berkovich projective line over $\mathcal{H}(x)$.  We will write $\mathbb{P}^{1,\text{an}}_x := \mathbb{P}^{1,\text{an}}_{\mathcal{H}(x)}$ for this fiber and $\infty_x$ for its point at infinity.  This description will allow us to consider dynamical systems corresponding to parameters in $\mathbb{A}^{1,\text{an}} \setminus \mathbb{A}^{1}(K)$.  The field $\mathcal{H}(x)$ is a complete non-archimedean field, but it is not algebraically closed in general, so we will need to discuss the structure of and dynamics on the Berkovich projective line over non-algebraically closed fields in \S \ref{sec:dynbpl}.

We will frequently encounter analytic sections of the projection $\pi_2$, i.e. analytic maps $\phi: U \longrightarrow \mathbb{P}^{1,\text{an}} \times U$ such that $\pi_2 \circ \phi = \text{id}$.  The following proposition gives a useful description of such maps.

\begin{prop} \label{prop:ansec}
Let $U \subseteq \mathbb{A}^{1,\text{an}}$ be an open set.  The map $\phi \longmapsto \pi_1 \circ \phi$ is a bijection from the set of analytic sections of the projection $\pi_2: \mathbb{P}^{1,\text{an}} \times U \longrightarrow U$ onto the set of analytic maps $U \longrightarrow \mathbb{P}^{1,\text{an}}$.  Moreover, for any such $\phi$, we have that $\phi(x) \in \mathbb{P}^{1}(\mathcal{H}(x))$ for all $x \in U$, viewing $\phi(x)$ as an element of the fiber $\mathbb{P}^{1,\text{an}}_x$ 
\end{prop}

\begin{proof}
Given $\psi: U \longrightarrow \mathbb{P}^{1,\text{an}}$ analytic and denoting $V := \psi^{-1}(\mathbb{A}^{1,\text{an}})$, we see that $\psi\restrict{V} \in \mathcal{O}_U(V)$.  For $x \in V$, we can then apply the canonical map $\mathcal{O}_U(V) \longrightarrow \mathcal{H}(x)$ to obtain an element $\phi_\psi(x) \in \mathcal{H}(x)$.  For $x \in U \setminus V$, set $\phi_\psi(x) = \infty_x$.  Note that on an affinoid domain $\mathcal{M}(\mathscr{A})$ of $V$, $\phi_\psi$ is just the analytic map $\mathcal{M}(\mathscr{A}) \longrightarrow \mathbb{A}^{1,\text{an}} \times \mathcal{M}(\mathscr{A})$ induced by the evaluation morphism
\begin{align*}
\mathscr{A}[z] &\longrightarrow \mathscr{A} \\
z &\longmapsto \psi\restrict{\mathcal{M}(\mathscr{A})}.
\end{align*}
Using this, one can check that $\phi_\psi: U \longrightarrow \mathbb{P}^{1,\text{an}} \times U$ is an analytic section of $\pi_2$.  It is also easy to see that the map $\psi \mapsto \phi_\psi$ is the inverse of the map in the statement of the proposition.  The ``Moreover'' statement follows immediately from the definition of $\phi_\psi$.
\end{proof}

To conclude this subsection, we discuss Berkovich space analogues of the inverse function theorem and the open mapping theorem that we will need later.  The following terminology will appear: a finite morphism of Berkovich spaces $Y \longrightarrow X$ is \textit{\'{e}tale} at a point $y \in Y$ if 
$\mathcal{O}_{Y,y} / \mathfrak{m}_{X,x} \mathcal{O}_{Y,y}$ is a finite separable field extension of $\kappa(x)$ and $\mathcal{O}_{Y,y}$ is a flat $\mathcal{O}_{X,x}$-module, where $x := \phi(y)$.

We will need the following fact about \'{e}tale morphisms.

\begin{prop} \label{prop:berketale}
Let $\phi: Y \longrightarrow X$ be a finite morphism of (good) Berkovich spaces.  Assume $\phi$ is \'{e}tale at a point $y \in Y$ and satisfies $\mathcal{H}(y) = \mathcal{H}(\phi(y))$.  Then, $\phi$ is a local isomorphism of Berkovich spaces at $y$.
\end{prop}

\begin{proof}
This is a special case of Theorem 3.4.1 in \cite{Ber93} and in fact appears as one of the first steps in the proof of this much more general theorem.
\end{proof}

\begin{lem} \label{lem:ift}
Let $U \subseteq \mathbb{A}^{1,\text{an}}$ be a connected open subset.  Let $g(\lambda,z)$ be a monic polynomial in $z$ whose coefficients are analytic functions of $\lambda$ defined on $U$, which we will view as an analytic function on $\mathbb{A}^{1,\text{an}} \times U$.  Let $\mathcal{C}$ be the Berkovich curve in $\mathbb{A}^{1,\text{an}} \times U \subseteq \mathbb{A}^{2,\text{an}}$ defined by the vanishing of $g(\lambda,z)$.  Then, $\pi_2: \mathcal{C} \longrightarrow U$ is locally an isomorphism of Berkovich spaces at any point $\xi \in \mathcal{C}$ where the multiplicity of $\pi_2$ is $1$.
\end{lem}

\begin{proof}
Denote $x := \pi_2(\xi)$.  In view of Proposition \ref{prop:berketale}, it suffices to show that $\pi_2 \restrict{\mathcal{C}}$ is \'{e}tale at $\xi$ and $\mathcal{H}(\xi) = \mathcal{H}(x)$.  For this, let $V := \mathcal{M}(\mathscr{A})$ be a closed affinoid neighborhood of $x$ in $U$.  Then, viewing $g(\lambda,z)$ as an element of $\mathscr{A} [ z ]$, we have that $\left( \pi_2 \restrict{\mathcal{C}} \right)^{-1}(V)$ is isomorphic to the Berkovich affinoid $\mathcal{M}(\mathscr{B})$, where $\mathscr{B} = \mathscr{A} [ z ] /(g)$.  Since $g$ is monic, we see that the algebra $\mathscr{B}$ is a free $\mathscr{A}$-module and hence is a flat $\mathscr{A}$-module.  Additionally, since $\pi_2: \mathcal{C} \longrightarrow U$ has multiplicity $1$ at $\xi$, we know that $\mathcal{O}_{\mathcal{M}(\mathscr{B}),\xi} = \mathcal{O}_{\mathcal{M}(\mathscr{A}),x}$, and in particular that
\[
\mathcal{O}_{\mathcal{M}(\mathscr{B}),\xi} / \mathfrak{m}_{\mathcal{M}(\mathscr{A}),x} \mathcal{O}_{\mathcal{M}(\mathscr{B}),\xi} = \kappa(x).
\]
We conclude that $\pi_2 \restrict{\mathcal{C}}$ is \'{e}tale at $\xi$.  The above computation also gives $\kappa(\xi) = \kappa(x)$ and hence $\mathcal{H}(\xi) = \mathcal{H}(x)$ after passing to completions.
\end{proof}

The open mapping theorem analogue that we need is the following result of Berkovich.
\begin{thm}[\cite{Ber90}*{Lemma 3.2.4}] \label{thm:openmap}
Let $\phi: Y \longrightarrow X$ be a finite morphism of pure dimensional $\mathbb{K}$-analytic spaces, and suppose that $X$ is locally irreducible and $\text{dim}(X) = \text{dim}(Y)$.  Then, $\phi$ is an open map.
\end{thm}
The definitions of dimension and irreducibility in the context of Berkovich spaces are analogous the corresponding notions in algebraic geometry.  See \cite{Ber90} for the precise formulations.

\subsection{The Berkovich Projective Line} \label{ssec:bpl}
The Berkovich projective line $\mathbb{P}^{1,\text{an}}_\mathbb{K}$ has many nice topological properties that make it especially useful for dynamics: it is a compact, Hausdorff, uniquely path-connected space that contains $\mathbb{P}^{1}(\mathbb{K})$ as a topological subspace; see \cite{BR10} or M. Jonsson's notes in \cite{DFN15} for a thorough treatment.  Note that unique path-connectedness implies that any connected subset is simply connected.  A base for its topology is formed by the collection of sets which are a connected component of the complement of a finite set, which we will call \textit{basic open sets}.  Denote the algebraic closure of $\mathbb{K}$ by $\mathbb{K}^a$ and its completion by $\widehat{\mathbb{K}^a}$.  The points $\xi \in \mathbb{P}^{1,\text{an}}_\mathbb{K}$ can be classified into four types:
\begin{itemize}
\item $\xi$ is of type I if $\mathcal{H}(\xi) \subseteq \widehat{\mathbb{K}^a}$.
\item $\xi$ is of type II if $\text{tr.deg}_{\widetilde{\mathbb{K} }} \widetilde{\mathcal{H}(\xi)} = 1$
\item $\xi$ is of type III if $\left\lvert \mathcal{H}(\xi)^\times \right\rvert \neq \left\lvert \mathbb{K}^\times \right\rvert$
\item $\xi$ is of type IV otherwise.
\end{itemize}
When $\mathbb{K}$ is algebraically closed, there is a simpler description: type I points are just the classical points of $\mathbb{P}^{1}(\mathbb{K})$, type II points are supremum norms over closed discs with radius in $\left\lvert \mathbb{K}^\times \right\rvert$, type III points are supremum norms over closed discs with radius not in $\left\lvert \mathbb{K}^\times \right\rvert$, and type IV points can be associated to nested sequences of closed discs with empty intersection.  We will use the notation $\zeta_{a,r}$ for the type II or III point associated to the closed disc $\overline{D}(a,r)$.  The point $\zeta_{0,1}$ is often called the \textit{Gauss point}.  There is a well-defined \textit{diameter map} $\xi \mapsto \text{diam}(\xi)$ that gives the radius of the disc associated to a type II or type III point, is $0$ at a type I point, and extends continuously along paths to type IV points.  To continue to make use of this description for $\mathbb{K}$ not algebraically closed, we note that the action of absolute Galois group $\text{Gal}(\mathbb{K}^\text{sep}/\mathbb{K})$ on $\mathbb{K}^a$ has a unique continuous extension to $\mathbb{P}^{1,\text{an}}_{\widehat{\mathbb{K}^a}}$.  Moreover, there is a canonical map $\pi: \mathbb{P}^{1,\text{an}}_{\widehat{\mathbb{K}^a}} \longrightarrow \mathbb{P}^{1,\text{an}}_\mathbb{K}$ given by quotienting by the action $\text{Gal}(\mathbb{K}^\text{sep}/\mathbb{K})$.  This quotient preserves types of points, and the diameter map also descends via this quotient.  See Section 3.9.1 of M. Jonsson's notes in \cite{DFN15}.

In view of the fact that $\mathbb{P}^{1,\text{an}}_\mathbb{K}$ is uniquely path-connected, we call the connected components of $\mathbb{P}^{1,\text{an}}_\mathbb{K} \setminus \left\lbrace \xi \right\rbrace$ the \textit{tangent directions} at $\xi$.  Type I and type IV points have a unique tangent direction, type III points have two tangent directions, and the set of tangent directions at a type II point can be naturally identified with the closed points of the scheme $\mathbb{P}^{1}_{\widetilde{\mathbb{K} }}$.  We will tend to write $\vec v$ for a tangent direction to emphasize that it is a direction, and use the notation $B(\vec v)$ when thinking of $\vec v$ as a subset of $\mathbb{P}^{1,\text{an}}_\mathbb{K}$.

Any analytic function $\psi$ defined on an open neighborhood of the Gauss point $\zeta_{0,1}$ has a reduction $\widetilde{\psi}$ which is a one-variable rational function defined over $\widetilde{\mathbb{K}}$.  If $\psi$ is a rational function, then $\widetilde{\psi}$ can be defined by multiplying the numerator and denominator by an element of $\mathbb{K}$ to ensure that all coefficients have absolute value less than or equal to $1$ and at least one coefficient has absolute value $1$, and then taking the reduction of each coefficient.  Otherwise, $\psi$ is a uniform limit of rational functions, so the reductions of a converging sequence of rational functions must stabilize.  A useful fact is that $\widetilde{\psi}$ is non-constant if and only if $\psi(\zeta_{0,1}) = \zeta_{0,1}$.  In this case, the action of $\psi$ on the tangent directions at $\zeta_{0,1}$ is given by the action of $\widetilde{\psi}$ on the closed points of the scheme $\mathbb{P}^{1}_{\widetilde{\mathbb{K}}}$.

In \S \ref{ssec:lamlemma}, we will also need some results about analytic functions on annuli.  For simplicity, we work over the algebraically closed field $K$ here.  An analytic function $\psi$ defined on an open annulus $D(a,s) \setminus \overline{D}(a,r) \subseteq \mathbb{A}^{1,\text{an}}$ can be expressed as a convergent Laurent series
\[
\psi(z) = \sum_{k = -\infty}^{\infty} b_k (z-a)^k
\]
with coefficients $b_k \in K$.  The \textit{inner Weierstrass degree} of $\psi$ is the largest integer $M$ such that $\left\lvert b_M \right\rvert r^M \geq \left\lvert b_k \right\rvert r^k$ for all $k$.  The \textit{outer Weierstrass degree} of $\psi$ is the smallest integer $N$ such that $\left\lvert b_N \right\rvert s^N \geq \left\lvert b_k \right\rvert s^k$ for all $k$.  The key property of these Weierstrass degrees is they allow us to count zeros.

\begin{prop} \label{prop:wdeg}
Suppose $\psi(z) = \sum_{k = -\infty}^{\infty} b_k (z-a)^k$ is an analytic function defined over $K$, not identically zero, on the open annulus $D(a,s) \setminus \overline{D}(a,r)$ with inner Weierstrass degree $M$ and outer Weierstrass degree $N$.  Then, $\psi$ has exactly $N-M$ zeros on the annulus counted with multiplicity.
\end{prop}

\begin{proof}
This follows from Theorem 1 in Section 2.2 of Chapter 6 in \cite{Rob00} together with the remark in Section 2.6 of Chapter 6 in \cite{Rob00} that this result also applies to Laurent series.  (The result is only stated for a complete algebraically closed extension of $\mathbb{Q}_p$, but the proof works for any complete algebraically closed non-archimedean field.) 
\end{proof}

It follows that if $\psi$ has no zeros on the annulus, then its inner and outer Weierstrass degrees must coincide.  We will refer to this common value as the Weierstrass degree in this situation.

Note that Proposition \ref{prop:wdeg} also applies to an analytic function $\psi$ on the disc $D(a,s)$, if we take the inner Weierstrass degree to be $0$ (take the limit as $r$ goes to $0$).

The set of type II and III points can also be equipped with a metric $d_{\mathbb{H}}$ called the \textit{hyperbolic metric} defined by
\[
d_{\mathbb{H}}\zeta_{a,r},\zeta_{b,s}) := 2 \log \max \left\lbrace r, s, \left\lvert a-b \right\rvert \right\rbrace -\log r -\log s.
\]
This metric can also be naturally extended to type IV points.  On the other hand, it would give infinite distance between type I points and is not compatible with the usual topology on $\mathbb{P}^{1,\text{an}}$ introduced above.  It will be important to understand how analytic functions interact with this metric.

\begin{prop} \label{prop:scaling}
Suppose $\psi(z) = \sum_{k = -\infty}^{\infty} b_k (z-a)^k$ is an analytic function on the open annulus $D(a,s) \setminus \overline{D}(a,r)$ as above, and $\psi-b_0$ has inner and outer Weierstrass degrees both equal to $N$.  Then,
\[
d_{\mathbb{H}}(\psi(\zeta_{a,r}),\psi(\zeta_{a,s})) = \max \left\lbrace N, -N \right\rbrace d_{\mathbb{H}}(\zeta_{a,r},\zeta_{a,s}).
\]
\end{prop}

\begin{proof}
Using the non-archimedean triangle inequality and the definition of Weierstrass degrees, we obtain that $\left\lvert \psi(z)-b_0 \right\rvert = \left\lvert b_N \right\rvert \left\lvert z \right\rvert^N$ for all $z$ in the annulus.  The result follows immediately.  This is also \cite{BR10}*{Theorem 9.46(C)} if $\psi$ is a rational function.
\end{proof}

\section{Dynamics on the Berkovich Projective Line} \label{sec:dynbpl}
This material can also be found in \cite{BR10} or M. Jonsson's notes in \cite{DFN15}.  For this subsection, let $f(z) \in \mathbb{K}(z)$ be a rational function.  Then, $f(z)$ naturally induces an analytic map from $\mathbb{P}^{1,\text{an}}_\mathbb{K}$ to itself: identifying $f(z)$ with a pair of homogeneous polynomials $[F(X,Y),G(X,Y)]$, we define
\[
\left\lvert P(X,Y) \right\rvert_{f(\xi)} := \left\lvert P\left(F(X,Y), G(X,Y) \right) \right\rvert_{\xi}
\]
for $\xi \in \mathbb{P}^{1,\text{an}}_\mathbb{K}$ and $P(X,Y) \in \mathbb{K}(X,Y)$ homogeneous.  We aim to study the iterates
\[
f^n := \underbrace{f \circ f \circ\ldots \circ f}_{n \text{ times}}
\]
of this map.  A point $\xi \in \mathbb{P}^{1,\text{an}}_\mathbb{K}$ is called a \textit{periodic point} for $f$ if $f^n(\xi) = \xi$ for some integer $n \geq 1$; its \textit{minimal period} is the smallest such $n$.  A \textit{fixed point} is a periodic point of minimal period $1$.  A type I fixed point $\xi$ is called \textit{attracting} if $\left\lvert f'(\xi) \right\rvert < 1$, \textit{indifferent} if $\left\lvert f'(\xi) \right\rvert = 1$, or \textit{repelling} if $\left\lvert f'(\xi) \right\rvert > 1$.  The value $\left\lvert f'(\xi) \right\rvert$ is called the \textit{multiplier} of the fixed point $\xi$.  A type II, III, or IV fixed point $\xi$ is called \textit{indifferent} if $f$ maps $\xi$ with multiplicity $1$ and is called \textit{repelling} otherwise.  The map $f$ is locally expanding near repelling fixed points, locally contracting near attracting fixed points, and neither contracting nor expanding near indifferent fixed points.  A periodic point $\xi$ of minimal period $n$ is called attracting, repelling or indifferent according to its classification as a fixed point of the rational function $f^n$.

The \textit{Julia set} $J(f)$ is a compact subset of $\mathbb{P}^{1,\text{an}}_\mathbb{K}$ that is totally invariant under $f$, consisting of the points at which the dynamics of $f$ are chaotic.  There are a variety of equivalent characterizations in the case where $\mathbb{K}$ is algebraically closed, e.g. it is the support of a canonical invariant measure and it is the complement of the domain of normality in an appropriate sense \cite{FKT12}.  It is also the topological closure of the set of all repelling periodic points, which will serve as a definition for our purposes.  In fact, it is the topological closure of the set of type I repelling periodic points if there is at least one type I repelling periodic point and $\text{char}(K) = 0$ by \cite{Bez01}*{Th\'{e}or\`{e}me 2}.  See Remark \ref{rmk:dense} for further discussion.  If $\mathbb{K}$ is not algebraically  closed, then one can view $f$ as a rational function $\widehat{f}$ with coordinates in $\widehat{\mathbb{K}^a}$ and define $J(f) \subseteq \mathbb{P}^{1,\text{an}}_\mathbb{K}$ to be the image of $J(\widehat{f})$ under the canonical projection $\pi: \mathbb{P}^{1,\text{an}}_{\widehat{\mathbb{K}^a}} \longrightarrow \mathbb{P}^{1,\text{an}}_\mathbb{K}$ described above.  One can check that $\text{deg}_\xi f = \text{deg}_{\widehat{\xi}} \widehat{f}$ and $\left\lvert f'(\pi(\xi)) \right\rvert = \left\lvert \widehat{f}'(\widehat{\xi}) \right\rvert$ for any fixed point $\widehat{\xi}$ of $\widehat{f}$.  Using this, it follows that the above definition of $J(f)$ coincides with the closure of the repelling periodic points of $f$ in $\mathbb{P}^{1,\text{an}}_\mathbb{K}$.

\section{Definitions and Preliminary Lemmas} \label{sec:deflem}
Our main object of study is an analytic family of rational maps.  For simplicity, we will focus only on families parametrized by an open connected subset of $\mathbb{A}^{1, \text{an}}$, although many definitions and proofs apply to families parametrized by arbitrary (good) Berkovich spaces.
\begin{defn}
Let $U \subseteq \mathbb{A}^{1, \text{an}}$ be a connected open set.  An \textit{analytic family of rational maps} parametrized by $U$ is an analytic map $\mathbb{P}^{1,\text{an}} \times U \longrightarrow \mathbb{P}^{1,\text{an}}$.  This is the same as a one-variable rational function with coefficients in $\mathcal{O}_U(U)$ such that after applying the canonical maps $\mathcal{O}_U(U) \longrightarrow \mathcal{H}(\lambda)$ for $\lambda \in U$, the induced rational functions with coefficients in $\mathcal{H}(\lambda)$ all have the same degree.  We will use the notation $\left\lbrace f_\lambda \right\rbrace_{\lambda \in U}$ to refer to the collection of rational functions with coefficients in $\mathcal{H}(\lambda)$, which we will view as analytic maps on the fibers
\[
f_\lambda: \mathbb{P}^{1,\text{an}}_\lambda \longrightarrow \mathbb{P}^{1,\text{an}}_\lambda.
\]
\end{defn}
Since each $f_\lambda$ is just a rational function defined over the complete (but not necessarily algebraically closed) field $\mathcal{H}(\lambda)$, the results from \S \ref{sec:dynbpl} apply.  We will write $J_\lambda$ for the Julia set of $f_\lambda$ in $\mathbb{P}^{1,\text{an}}_\lambda$ and write $f'_\lambda$ for the derivative with respect to the dynamical coordinate $z$.  In an abuse of notation, in order to preserve the terminology from \S \ref{sec:dynbpl} , we will call points in $\mathbb{P}^{1,\text{an}} \times U$ type I, II, III, or IV if they have this classification in their fiber $\mathbb{P}^{1, \text{an}}_\lambda$.  We will occasionally refer to this as the fiberwise type of a point when the danger of confusion is particularly acute.

\begin{defn}
Let $x \in U$ be given.  A type I periodic point $\xi \neq \infty_x$ of $f_x$ of minimal period $n$ will be called $\textit{unstably indifferent}$ if ``its multiplier is the Gauss point'' in the sense that $(f_\lambda^n)'(\xi)=\zeta_{0,1}$, viewing $(f_\lambda^n)'$ as a map from $\mathbb{P}^{1,\text{an}} \times U$ to $\mathbb{P}^{1, \text{an}}$.  If $\xi = \infty_x$, then it will be called unstably indifferent if it satisfies the above condition after a change coordinates.
\end{defn}
Since $\zeta_{0,1}$ is the unique boundary point of the Berkovich disc $\overline{D}(0,1)$, unstably indifferent periodic points signal a change of the classification (attracting, repelling, or indifferent) of a periodic point.

\begin{example}
Consider the family $f_\lambda(z)=z^2+\lambda z$.  Then, $f_\lambda'(z) = 2z+\lambda$, so $f_\lambda': \mathbb{P}^{1,\text{an}} \times U \longrightarrow \mathbb{P}^{1,\text{an}}$ sends the point $0 \in \mathbb{P}^{1,\text{an}}_{\zeta_{0,1}} \subseteq \mathbb{P}^{1,\text{an}} \times U$ to $\zeta_{0,1} \in \mathbb{P}^{1,\text{an}}$ and hence this fixed point is unstably indifferent.  A different parametrization of this family will be discussed further in Example \ref{ex:quadratics}.
\end{example}

An important tool in our study of the dynamics of a family of rational maps are the Berkovich $n$-period curves, which parametrize the points of $\left\lbrace f_\lambda \right\rbrace$ with minimal period dividing $n$.
\begin{defn}
Let $A$ be an arbitrary field.  Let $f \in A(z)$ and $n \in \mathbb{N}$ be given.  Represent $f^n$ by homogeneous polynomials $[F_n(X,Y), G_n(X,Y)]$ in $A[X,Y]$.  We define the \textit{$n$-period polynomial} for $f$ to be the homogeneous polynomial
\[
\Phi_{f,n}(X,Y) := Y F_n(X,Y)-X G_n(X,Y) \in 
A[X,Y].
\]
The \textit{Berkovich $n$-period curve} for $\left\lbrace f_\lambda \right\rbrace$, denoted by $\mathcal{C}_n(\left\lbrace f_\lambda \right\rbrace)$, is defined as the vanishing locus of $\Phi_{\left\lbrace f_\lambda \right\rbrace,n}(X,Y)$ in $\mathbb{P}^{1,\text{an}} \times U$,  viewing $\left\lbrace f_\lambda \right\rbrace$ as a rational function with coefficients in $\mathcal{O}_U(U)$ and taking the homogeneous coordinate view of $\mathbb{P}^1$.
\end{defn}

In order to apply Lemma \ref{lem:ift}, we will need the fact that the Berkovich $n$-period curves can be cut out locally by a monic polynomial:

\begin{lem} \label{lem:monic}
For any $\xi \in \mathcal{C} := \mathcal{C}_n(\left\lbrace f_\lambda \right\rbrace)$, there is an affinoid neighborhood $V := \mathcal{M}(\mathscr{A})$ of $x := \pi_2(\xi)$ in $U$ such that $\mathcal{C} \cap (\mathbb{A}^{1,\text{an}} \times V)$ is the vanishing locus of a monic polynomial in $\mathscr{A}[z]$ and $\xi \in \mathbb{A}^{1,\text{an}} \times V$, after a change of coordinates on $\mathbb{P}^{1,\text{an}} \times V$.
\end{lem}

\begin{proof}
There are only finitely many points of period dividing $n$ for $f_x$, so we can change variables so that none of them are $\infty_x$. Then, writing
\[
\Phi_{\left\lbrace f_\lambda \right\rbrace, n} (X,Y) = a_0(\lambda) X^N +a_1(\lambda)X^{N-1}Y+\ldots+ a_N(\lambda) Y^N
\]
with all $a_i \in \mathcal{O}_U(U)$ and no common roots in $U$, we have that $a_0(x) \neq 0$.  By continuity  of $a_0(\lambda)$, the subset $\left\lbrace y \in U \mid a_0(y) \neq 0 \right\rbrace$ is open and hence it contains an affinoid subdomain $V := \mathcal{M}(\mathscr{A})$ which also contains $x$.  Now viewing the $a_i \in \mathscr{A}$, we observe that $a_0$ is invertible in $\mathscr{A}$, and hence we can take
\[
g(z) = z^n+\frac{a_1}{a_0}z^{n-1}+\ldots \frac{a_N}{a_0} \in \mathscr{A}[z]
\]
as the desired monic polynomial.
\end{proof}

\begin{defn}
The \textit{multiplicity} of a type I periodic point $\xi$ of minimal period $n$ is its multiplicity under the map
\[
\pi_2: \mathcal{C}_n(\left\lbrace f_\lambda \right\rbrace) \longrightarrow U.
\]
\end{defn}

As in the complex case, type I repelling periodic points at type I parameters always have multiplicity $1$.  However, there is a new bifurcation possible in non-archimedean dynamics: type I repelling periodic points can collide at a non-classical parameter.  Moreover, this collision can be detected intrinsically in the fiber $\mathbb{P}^{1,\text{an}}_\lambda$ without any information about the family $\left\lbrace f_\lambda \right\rbrace$.  These results are described more precisely in the following lemma.

\begin{lem} \label{lem:ppcol}
Let $x \in U$ and let $\xi$ be a type I repelling periodic point of $f_x$ of minimal period $n$.  View $\xi$ as an element of $\mathcal{C} := \mathcal{C}_n(\left\lbrace f_\lambda \right\rbrace)$ and denote by $i_x(\xi)$ the canonical inclusion of $\xi$ in $\mathbb{P}^{1,\text{an}}_x$.  Denote by $m$ the multiplicity of $\xi$ as a periodic point.  Then,
\begin{enumerate}[$(1)$]
\item $m = [ \mathcal{H}(i_x(\xi)) : \mathcal{H}(x) ]$
\item $m = 1$ if $x$ is a type I point.
\end{enumerate}
\end{lem}

\begin{proof}
Change variables so that $\infty_x$ is not a periodic point of period dividing $n$ for $f_x$.  As in the proof of Lemma \ref{lem:monic}, choose an affinoid neighborhood $V := \mathcal{M}(\mathscr{A})$ of $x$ in $U$ such that there is an isomorphism of Berkovich spaces $\mathcal{C} \cap \left( \mathbb{A}^{1,\text{an}} \times V \right) \cong \mathcal{M}(\mathscr{B})$, where $\mathscr{B} = \mathscr{A}[z]/(h(z))$
and $h(z)$ is obtained from the polynomial $\Phi_{\left\lbrace f_\lambda \right\rbrace,n}(z) := \Phi_{\left\lbrace f_\lambda \right\rbrace,n}(z,1) \in \mathcal{O}_U(U)$ by applying the natural map $\mathcal{O}_U(U) \longrightarrow \mathscr{A}$ to the coefficients.  Thus,
\[
\mathcal{O}_{\mathcal{C},\xi} / \mathfrak{m}_{U,x} \mathcal{O}_{\mathcal{C},\xi} \cong \big(\kappa(x)[z]/ (\Phi_{f_x,n}(z)) \big)_{(P(z))}, 
\]
for some irreducible factor $P(z)$ of $\Phi_{f_x,n}(z)$ in $\kappa(x)[z]$.  Recalling that
\[
m = \text{dim}_{\kappa(x)} \left( \mathcal{O}_{\mathcal{C},\xi} / \mathfrak{m}_{U,x} \mathcal{O}_{\mathcal{C},\xi} \right),
\]
this shows that $m = k \cdot \text{deg}(P)$, where $k$ is the number of times that $P$ appears as a factor of $\Phi_{f_x,n}(z)$.  Noting that
\[
\Phi_{f_x,n}'(z) = \left( (f_x^n)'(z)-1 \right)G_n(z)+\left( f_x^n(z)-z \right) G_n'(z),
\]
we see that the multiplier of a multiple root of $\Phi_{f_x,n}(z)$ must be a root of unity and therefore can only occur at an indifferent periodic point.  Thus, $k =1$ and $\kappa_\mathcal{C}(\xi) / \kappa(x)$ is a separable field extension.  We compute
\[
m = \text{deg}(P) = [\kappa_{\mathcal{C}}(\xi): \kappa(x) ] = [ \mathcal{H}_{\mathcal{C}}(\xi) :  \mathcal{H}(x) ],
\]
using Propositon \ref{prop:weaklystab} for the last equality.  In particular, $P(z)$ remains irreducible in $\mathcal{H}(x)[z]$, which implies that $i_x(\xi)$ is the unique point of $\mathbb{P}^{1,\text{an}}_x$ that satisfies $\left\lvert P(z) \right\rvert_{i_x(\xi)} = 0$.  Therefore, $m = [\mathcal{H}(i_x(\xi)) : \mathcal{H}(x) ]$, proving (1).

For (2), we observe that $\mathcal{H}(x) \cong K$ is algebraically closed for a type I parameter $x$.
\end{proof}

\begin{defn}
Let $\left\lbrace f_\lambda \right\rbrace$ be an analytic family of rational maps parametrized by a connected open set $U \subseteq \mathbb{A}^{1,\text{an}}$.  Assume that the Julia set $J_{x_0}$ contains a type I point for some $x_0 \in U$.  We say that the Julia sets $J_\lambda$ \textit{move analytically} on $U$ if there is a family of homeomorphisms
\[
\phi_\lambda: J_{x_0} \longrightarrow J_{\lambda}
\]
such that
\begin{enumerate}
\item $\phi_{x_0} = \text{id}$
\item the map $\lambda \mapsto \phi_\lambda(\xi)$ from $U$ to $\mathbb{P}^{1,\text{an}} \times U$ is analytic for each type I point $\xi \in J_{x_0}$
\item $f_\lambda \circ \phi_\lambda = \phi_\lambda \circ f_{x_0}$ for all $\lambda \in U$.
\end{enumerate}
\end{defn}

\begin{rmk} \label{rmk:dense}
Note that this definition is only sensible if the type I points are dense in $J_x$, because otherwise there are no conditions on how the non-type I points of the Julia set move.  If $J_x$ has any type I points, then they form a dense subset, because the backward orbit of any point in $J_x$ is dense \cite{BR10}*{Corollary 10.58}.  By a Theorem of B\'{e}zivin \cite{Bez01}*{Th\'{e}or\`{e}me 2}, if $\text{char}(K)=0$ and there is at least one type I repelling periodic point in $J_x$, then the type I repelling periodic points are dense in $J_x$.  It is conjectured that the same is true as long as $\text{char}(K)=0$ and $J_x$ has any type I points.  The characteristic $p$ case is more subtle.
\end{rmk}

\begin{rmk}
An analytic motion of Julia sets does not necessarily preserve types of points, although by definition type I points will always be kept type I by the motion.  In \cite{Ben02}, R. Benedetto provides an example of a one-parameter family of polynomials and a region in parameter space which exhibits wandering domains for a dense subset of type I parameters in this region and no wandering domains for a different dense subset of type I parameters in this region.  If there is a subset of this region where the Julia sets move analytically, then the analytic motion would change points from type II or III to type IV.  In complex dynamics, the $J$-stable parameters are dense, so it is likely that such a subset exists.
\end{rmk}

To conclude this section, we provide an example to help illustrate these definitions and Theorem \ref{thm:main}.

\begin{example} \label{ex:quadratics}
Consider the family of quadratic polynomials $z^2+\lambda$ with $\lambda \in \mathbb{A}^{1,\text{an}}$.  For simplicity, we will assume that $K$ does not have residue characteristic $2$.  In complex dynamics, the analogous family exhibits extremely intricate behavior and has been the subject of extensive study.  In non-archimedean dynamics, the behavior of this family is much easier to understand.  For $\left\lvert \lambda \right\rvert \leq 1$, the Julia set $J_\lambda$ consists of a single type II repelling fixed point at the Gauss point $\zeta_{0,1}$.  For $\lambda \in \mathbb{A}^{1,\text{an}}$ satisfying $\left\lvert \lambda \right\rvert > 1$ and $\text{diam}(\lambda) < \left\lvert \lambda \right\rvert$, the Julia set $J_\lambda$ consists only of type I points and is homeomorphic to a Cantor set.  In fact, $J_\lambda$ is homeomorphic to the space of sequences on two letters $\left\lbrace 0, 1 \right\rbrace^{\mathbb{N}}$ and under this homeomorphism $z^2+\lambda$ acts as the left shift map.  

Thus, there is a dramatic bifurcation at $\lambda = \zeta_{0,1}$.  Indeed, every type I periodic point is unstably indifferent at this parameter.  There are also more subtle bifurcations at every point along the path in parameter space $\left\lbrace  \lambda \in \mathbb{A}^{1,\text{an}} \mid \text{diam}(\lambda) =\left\lvert \lambda \right\rvert  > 1 \right\rbrace$ between $\zeta_{0,1}$ and $\infty$ noninclusive.  At these parameter values, every type I periodic point is repelling and has multiplicity $2$.  The Julia set can be identified with a quotient of the space of sequences on two letters by the action of switching the two letters.  These bifurcations are obstructions to an analytic motion of Julia sets on the open set $\left\lbrace \left\lvert \lambda \right\rvert > 1 \right\rbrace$, even though the dynamical behavior is the same away from this path.  One way to think about this is that for $x \in \mathbb{A}^{1,\text{an}}$ satisfying $\left\lvert x \right\rvert > 1$ and $\text{diam}(x) < \left\lvert x \right\rvert$, viewing $J_x$ and $J_{-x}$ as spaces of sequences on two letters, there is no canonical way to identify the two letters for $J_x$ with the two letters for $J_{-x}$.  Note that this phenomenon also occurs in the complex case but does not prevent the Julia set from moving holomorphically on the complement of the Mandelbrot set.  Moving holomorphically along two non-homotopic paths from $x$ to $-x$, which is possible since the complement of the Mandelbrot set is not simply connected, will give the two different identifications of the letters for the Julia sets.  In the non-archimedean case, every connected subset of $\mathbb{A}^{1,\text{an}}$ is simply connected, so this cannot occur.

In this example, one can verify directly that the two conditions in Theorem \ref{thm:main} are equivalent on every connected open set $U \subseteq \mathbb{A}^{1,\text{an}}$.  The main idea is that for any $\left\lvert x \right\rvert > 1$, two analytic inverses of $f_{\lambda}$ defined near $0$ can be constructed with coefficients in $\bigcap_{r < \left\lvert x \right\rvert} K \left\lbrace r^{-1} (\lambda-x) \right\rbrace$, using the fact that $\lambda$ has a square root in $\bigcap_{r < \left\lvert x \right\rvert} K \left\lbrace r^{-1} (\lambda-x) \right\rbrace$.  Applying infinite sequences of these inverses allows one to construct an analytic motion of the Julia sets on $D(x, \left\lvert x \right\rvert)$, and a similar argument can be used to show that the type I repelling periodic points have multiplicity $2$ along the segment from $\zeta_{0,1}$ to $\infty$.
\end{example}

\section{Main Results}
\subsection{The Non-archimedean $\lambda$-Lemma} \label{ssec:lamlemma}
\lamlemma*

\begin{note}
We will sometimes to a family of injections satisfying the hypotheses of the Non-archimedean $\lambda$-Lemma as an analytic motion.
\end{note}

The full proof of the Non-archimedean $\lambda$-Lemma is very technical and unnecessary to understanding the rest of the paper, so we have isolated it and the required lemmas to this subsection.  On the other hand, the proof that the motion can be extended to type I points of $\overline{E}$ in the case that $U$ is a disc is quite simple.  We give a brief sketch of the argument here:

\begin{proof}[Proof sketch of the disc case]
After a change of variables analytic in $\lambda$, we may assume that $\infty \in E$ and $\infty$ is fixed by the motion.  Denote $\psi_{\alpha}(\lambda) := (\pi_1 \circ \phi_\lambda)(\alpha)$ for $\alpha \in E \setminus \left\lbrace \infty \right\rbrace$ and $\lambda \in U$.  Then, each $\psi_{\alpha}$ must avoid $\infty$ by the injectivity of the $\phi_{\lambda}$ and hence is an analytic function from $U$ to $\mathbb{A}^{1,\text{an}}$.  Moreover, again using the injectivity of the $\phi_{\lambda}$, we have that $\psi_{\alpha}-\psi_{\beta}$ must avoid $0$ for each distinct $\alpha,\beta \in E \setminus \left\lbrace \infty \right\rbrace$.  Since $\psi_{\alpha}-\psi_{\beta}$ is analytic and $U$ is a disc, we have that $(\psi_{\alpha}-\psi_{\beta})(U)$ is a disc as well.  This disc avoids $0$, so $\left\lvert \psi_{\alpha}-\psi_{\beta} \right\rvert$ is constant on $U$.  We conclude that 
\[
\left\lvert \phi_{\lambda}(\alpha)-\phi_{\lambda}(\beta) \right\rvert_{\lambda} = \left\lvert (\psi_{\alpha}-\psi_{\beta})(\lambda) \right\rvert = \left\lvert (\psi_{\alpha}-\psi_{\beta})(x_0) \right\rvert = \left\lvert \alpha-\beta \right\rvert_{x_0}
\]
for all $\lambda \in U$ and distinct $\alpha, \beta \in E \setminus \left\lbrace \infty \right\rbrace$.  This shows that each $\phi_\lambda$ is uniformly continuous viewing  $E \setminus \left\lbrace \infty \right\rbrace \subseteq \mathcal{H}(x_0)$ and hence can be extended continuously to $\overline{E} \cap \mathcal{H}(x_0)$.
\end{proof}

The first lemma we will need for the full proof is a consequence of the $S$-unit equation over function fields.

\begin{thm}[\cite{Ros02}*{Theorem 7.19}] \label{thm:sunit}
Let $L$ be a function field with a perfect constant field $F$.  Let $S$ be a finite set of primes of $L$. Then, there are only finitely many pairs
of separable, non-constant $S$-units $(u, v)$ such that $u + v = 1$. ($u$ is said to be separable if the field extension $L / F (u)$ is separable). If the characteristic of $F$ is zero, then every solution is separable. If the characteristic of $F$ is $p > 0$, then the most general solution to $X + Y = 1$ in non-constant $S$-units is $(u^{p^m} , v^{p^m} )$ where $(u, v)$ is a separable, non-constant solution in $S$-units and $m \in \mathbb{Z}$, $m \geq 0$.
\end{thm}

\begin{lem} \label{lem:sunitdiff}
Let $L$ be a function field with perfect constant field $F$.  Let $S$ be a finite set of primes of $L$.  Suppose $A \subseteq L$ is a set of $S$-units such that $u-v$ is also an $S$-unit for all distinct $u,v \in A$.  Let $u_0 \in A$ be given, and denote $B := \left\lbrace u/u_0 \mid u \in A \right\rbrace$.  Then, either $B$ is a finite set or $B \subseteq F$.  In particular, $sup_{u \in A} \text{deg}(u) < \infty$.
\end{lem}

\begin{proof}
Denote $B' := \left\lbrace b \in B \text{ non-constant } \right\rbrace$.  Our first step is to show that $B'$ is finite.  Suppose for the sake of a contradiction that $B'$ is infinite.  Observe that its elements satisfy
\[
\frac{u}{u_0}+\frac{u_0-u}{u_0} = 1.
\]
and that $u/u_0$ and $(u_0-u)/u_0$ are non-constant $S$-units.  If $\text{char}(L)=0$, then Theorem \ref{thm:sunit} tells us that $B'$ is contained in the finite set $C$ of separable, non-constant solutions to $X+Y = 1$ in $S$-units, and we are finished with this step.  Otherwise, $\text{char}(L)=p>0$ and Theorem \ref{thm:sunit} tells us that $B'$ consists of $p^m$-th powers of elements of $C$.  By the pigeonhole principle, there is an element $v \in C$ and a strictly increasing sequence $\left\lbrace m_i \right\rbrace_{i=1}^{\infty}$ of positive integers such that $v^{p^{m_i}} \in B'$ for all $i \geq 1$.  By definition of $B'$, there exist $u_i \in A$ such that $v^{p^{m_i}} = u_i/u_0$ for all $i \geq 1$.  Choose $i$ large enough that $-1+p^{m_i-m_1} > \max_{w \in C} \text{deg}(w)$.  Then,
\[
\frac{u_i}{u_1} = \frac{u_i}{u_0} \frac{u_0}{u_1} = v^{p^{m_i}} v^{-p^{m_1}} = \left( v^{-1+p^{m_i-m_1}} \right)^{p^{m_1}}.
\]
Moreover, $u_i/u_1$ and $(u_1-u_i)/u_1$ are non-constant $S$-units satisfying
\[
\frac{u_i}{u_1}+\frac{u_1-u_i}{u_1} = 1,
\]
so $u_i/u_1$ is a $p^m$-th power of an element of $C$ by Theorem \ref{thm:sunit}.  Since $-1+p^{m_i-m_1}$ is relatively prime to $p$, we conclude that $v^{-1+p^{m_i-m_1}} \in C$.  However, by construction
\[
\text{deg}(v^{-1+p^{m_i-m_1}}) \geq -1+p^{m_i-m_1} > \max_{w \in C} \text{deg}(w),
\]
which gives the desired contradiction.

Therefore, $B'$ is finite.  If $B'$ is empty, then $B \subseteq F$ and we are finished.  Otherwise, there is some non-constant element $w/u_0 \in B$.  Note that for any constant $u/u_0 \in B$, we have that $\frac{w}{u} = \frac{w}{u_0} \left( \frac{u}{u_0} \right)^{-1}$ so $\text{deg}(w/u) = \text{deg}(w/u_0)$ and in particular $w/u$ is non-constant. Thus, $w/u$ and $(u-w)/u$ are non-constant $S$-units that satisfy
\[
\frac{w}{u}+\frac{u-w}{u} = 1
\]
and hence $w/u$ is also a $p^m$-th power of an element of $C$.  Since $C$ is finite and all such $w/u$ have the same degree, we conclude that there are only finitely many such $w/u$, and hence there are only finitely many constants $u/u_0 \in B$.
\end{proof}

We collect some facts about analytic functions on basic open sets that we will need in the following proposition.

\begin{prop} \label{prop:basicopen}
Let $U$ be an open subset of $\mathbb{P}^{1,\text{an}}$ of the form
\[
U = D(0,1/r) \setminus \left( \coprod_{i=1}^n \overline{D}(a_i,r) \right)
\]
with $r<1$ and each $a_i$ in a different finite tangent direction at the Gauss point $\zeta_{0,1}$.  Let $\psi$ be an analytic function on $U$ with no zeros on $U$.  Choose $\sigma \in K$ with $\left\lvert \sigma \right\rvert = \left\lvert \psi(\zeta_{0,1}) \right\rvert$ and denote by $d$ the degree of the rational function $\widetilde{\psi/\sigma}$.  Denote by $[0,\infty]$ the segment from $0$ to $\infty$ in $\mathbb{P}^{1,\text{an}}$.  Then, $\psi$ has the following properties:
\begin{enumerate}[(a)]
\item $\psi$ has a unique continuous extension to $\overline{U}$.
\item Either $\psi(U)$ is contained in a connected component of $\mathbb{P}^{1,\text{an}} \setminus [0,\infty]$ or $\psi(\zeta_{0,1}) \in [0,\infty]$.
\item The Weierstrass degree of $\psi$ on an annulus contained in $U$ with one endpoint at $\zeta_{0,1}$ in the direction of a point $a \in U(K)$ is the order of vanishing of $\widetilde{\psi/\sigma}$ at $\widetilde{a}$ (negative if $\widetilde{a}$ is a finite pole or if $\widetilde{a} = \widetilde{\infty}$ is a zero).
\item $\psi(\overline{U}) \subseteq \overline{D}(0,R)$ where $R = \left\lvert \psi(\zeta_{0,1}) \right\rvert r^{-d}$.   
\end{enumerate}
\end{prop}

\begin{proof}
Let $1 \leq i \leq n$ be given.  Consider the Laurent series expansion $\psi(z) = \sum_{k=-\infty}^{\infty} b_{i,k} (z-a_i)^k$ on the annulus $D(0,1) \setminus \overline{D}(a_i,r)$.  Choose $u_i$ with $r < u_i \leq 1$ small enough that $\psi-b_{i,0}$ has no zeros on the annulus $D(a_i,u_i) \setminus \overline{D}(a_i,r)$.  Then, its inner and outer Weierstrass degrees must coincide on this annulus; say both are equal to $N_i$.  From the non-archimedean triangle inequality, we have $\left\lvert \psi(z)-b_{i,0} \right\rvert = \left\lvert b_{i,N_i} \right\rvert \left\lvert z-a_i \right\rvert^{N_i}$ for $z \in D(a_i,u_i) \setminus \overline{D}(a_i,r)$ and hence $\psi(\zeta_{a_i,t}) = \zeta_{b_{i,0}, \left\lvert b_{i,N_i} \right\rvert t^{N_i}}$ for $r< t \leq u_i$.

Define $\psi(\zeta_{a_i,r}) = \zeta_{b_{i,0}, \left\lvert b_{i,N_i} \right\rvert r^{N_i}}$.  An analogous argument can be used to extend $\psi$ to $\zeta_{0,1/r}$, and it is easy to check that this extension of $\psi$ to $\overline{U}$ is continuous, proving (a).

For (b), assume that $\psi(U)$ is not contained in a connected component of $\mathbb{P}^{1,\text{an}} \setminus [0,\infty]$.  This means that $\left\lvert \psi \right\rvert$ is not constant and hence there is some direction at $\zeta_{0,1}$ such that $\psi$ has non-zero Weierstrass degree on an annulus with one endpoint at $\zeta_{0,1}$ in that direction.  Since $\psi$ has no zeros, this direction can only be towards one of the $a_i$ or towards $\infty$.  Assume that it is towards some $a_i$; the case of $\infty$ is analogous.  Then, we have that $\left\lvert b_{i,N_i} \right\rvert > \left\lvert b_{i,0} \right\rvert$ and we can take $t = u_i =1$ in the above computation.  This yields 
\[
\psi(\zeta_{0,1}) = \psi(\zeta_{a_i,1}) = \zeta_{b_{i,0}, \left\lvert b_{i,N_i} \right\rvert} = \zeta_{0, \left\lvert b_{i,N_i} \right\rvert} \in [0,\infty].
\]

For (c), it suffices to consider the case where $\psi$ is a rational function and $\left\lvert \psi(\zeta_{0,1}) \right\rvert = 1$.   First, replace $\psi$ by $\psi/\sigma$, and then use the fact that a general analytic $\psi$ has the form $\psi = \psi_1+\text{error}$, where $\psi_1$ is a rational function, $\widetilde{\psi}=\widetilde{\psi_1}$, and $\psi$ and $\psi_1$ have the same Weierstrass degrees on each annulus contained in $U$ with one endpoint at $\zeta_{0,1}$.  (The statement about reductions follows immediately from the definition in terms of limits of rational functions, and it is easy to check that adding a Laurent series with small coefficients will not affect the Weierstrass degrees on an annulus.)  For the rational function case, let $a,b \in K$ with $\left\lvert a \right\rvert, \left\lvert b \right\rvert \leq 1$ be given.  If $\widetilde{a}=\widetilde{b}$, then $\left\lvert b-a \right\rvert < 1$ and it is easy to check that multiplying a Laurent series $\sum_{k=-\infty}^{\infty} b_k (z-a)^k$ with no zeros on $D(0,1) \setminus \overline{D}(a,r)$ by $z-b = (z-a)-(b-a)$ will increase the outer Weiertrass degree by $1$.  It follows that multiplying such a Laurent series by $(z-b)^{-1}$ will decrease the outer Weierstrass degree by $1$.  If $\widetilde{a} \neq \widetilde{b}$, then $\left\lvert b-a \right\rvert = 1 > \left\lvert z-a \right\rvert$ for $z$ in the above annulus and multiplying a Laurent series of the above form by $z-b = (z-a)-(b-a)$ will leave the outer Weierstrass degree unchanged, and hence the same holds for $(z-b)^{-1}$.  Similarly, since $\left\lvert b \right\rvert = 1 < \left\lvert z\right\rvert$ for $z \in D(0,1/r) \setminus \overline{D}(0,1)$, it follows that multiplying a Laurent series with no zeros on this annulus by $z-b$ will increase the inner Weierstrass degree by $1$ and multiplying by $(z-b)^{-1}$ will decrease the inner Weierstrass degree by $1$.  Starting with a constant function with absolute value equal to $1$ and repeatedly applying these computations shows (c) for a rational function $\psi$ with no zeros on $U$ and $\left\lvert \psi(\zeta_{0,1}) \right\rvert =1$. 

For (d), note that $\left\lvert \psi \right\rvert$ is constant equal to $\left\lvert \psi(\zeta_{0,1}) \right\rvert$ on every disc contained in $U$ with $\zeta_{0,1}$ as its boundary point and on every annulus $D(0,1) \setminus \overline{D}(a_i,r)$ where $\psi$ has Weierstrass degree $0$.  For the remaining directions, the above computation (together with the analogous computation for the direction towards $\infty$) shows that the maximal value of $\left\lvert \psi \right\rvert$ on $\overline{U}$ occurs at one of the boundary points of $\overline{U}$ and is equal to
\[
\max_{i} \left\lvert b_{i,N_i} \right\rvert r^{N_i} = \max_i \left\lvert \psi(\zeta_{0,1}) \right\rvert r^{N_i} \leq \left\lvert \psi(\zeta_{0,1}) \right\rvert r^{-d},
\]
using (c) for the last inequality.
\end{proof}

\begin{rmk}
Proposition \ref{prop:basicopen}(c) in the case that $\psi$ is a rational function is essentially \cite{RL03}*{Proposition 3.3} or \cite{BR10}*{Corollary 9.25}, since the directional multiplicity of $\psi$ at $\zeta_{0,1}$ in the direction towards some $a \in U$ in the terminology of \cite{BR10} is the same as the Weierstrass degree of $\psi$ on a small annulus towards $a$ with one endpoint at $\zeta_{0,1}$ if this Weierstrass degree is positive. 
\end{rmk}

We also need the following more technical lemma about analytic functions on basic open sets.
\begin{lem} \label{lem:injective}
Let $U$ be an open subset of $\mathbb{P}^{1,\text{an}}$ of the form
\[
U = D(0,1/r) \setminus \left( \coprod_{i=1}^n \overline{D}(a_i,r) \right)
\]
with $r<1$ and each $a_i$ in a different finite tangent direction at the Gauss point $\zeta_{0,1}$.  Let $x_0 \in U$ be given, and denote $r_0 := \min_{1 \leq i \leq n} \left\lbrace \left\lvert T - a_i \right\vert_{x_0}, 1/ \left\lvert x_0 \right\rvert \right\rbrace$.  Let $\psi_1$ and $\psi_2$ be analytic functions on $U$, and denote $t_j = \left\lvert \psi_j(x_0) \right\vert$ for $j=1,2$.  Suppose $\psi_1$ and $\psi_2$ satisfy
\begin{enumerate}[$(1)$]
\item $\psi_1$, $\psi_2$, and $\psi_1-\psi_2$ have no zeros on $U$
\item $\max(r_0,r/r_0) t_1 < t_2 < t_1$ if $r_0<1$
\item $r t_1< t_2 <t_1$ if $r_0=1$.
\end{enumerate}
Then, $\displaystyle \frac{\left\lvert \psi_2(x) \right\rvert}{\left\lvert \psi_1(x) \right\rvert} = \frac{t_2}{t_1}$ for all $x \in U$.
\end{lem}

\begin{proof}
For each direction $\vec{v}$ at $\zeta_{0,1}$, denote by $A_{\vec{v}}$ the largest open annulus or disc with one endpoint at $\zeta_{0,1}$ in the direction $\vec{v}$ that is contained in $U$.  Note that by (1), we have that $\psi := \psi_2/\psi_1$ is analytic on $U$ and omits the values $0$ and $1$ on $U$.  Denote $t := t_2/t_1$.  Since $\psi$ has no zeros or poles on $U$, it follows that it has the same inner and outer Weierstrass degree on each $A_{\vec{v}}$ by Proposition \ref{prop:wdeg}.  Denote by $d_{\vec{v}}$ the Weierstrass degree of $\psi$ on $A_{\vec{v}}$.

\textbf{Step 1: Prove the result in the special case $x = \zeta_{0,1}$}.  If $x_0 = \zeta_{0,1}$, then this follows immediately.  Now, assume that $x_0 \neq \zeta_{0,1}$, and denote by $\vec{w}$ the direction at $\zeta_{0,1}$ towards $x_0$.  It suffices to show that $d_{\vec{w}} = 0$, since this implies that $\left\lvert \psi \right\rvert$ is constant on $\overline{A}_{\vec{w}}$.  If $A_{\vec{w}}$ is a disc, then this follows immediately from the fact that $\psi$ has no zeros.  Otherwise, $\vec{w}$ is the direction towards some $a_i$ or $\infty$, and hence $r_0<1$ and the inequalities from (2) apply.  We will assume that $\vec{w}$ is in the direction of some $a_i$; the case of $\infty$ is analogous and can be reduced to this case by a change of coordinates on $U$.  This means that $\psi \restrict{A_{\vec{w}}}$ has the form
\[
\psi(y) = \sum_{k= -\infty}^{\infty} b_k (y-a_i)^k.
\]
The condition that the inner and outer Weierstrass degrees on this annulus are both $d_{\vec{w}}$ means that $\left\lvert b_{d_{\vec{w}}} \right\rvert > \left\lvert b_k \right\rvert$ and $\left\lvert b_{d_{\vec{w}}} \right\rvert r^{d_{\vec{w}}} > \left\lvert b_k \right\rvert r^k$ for all $k \neq d_{\vec{w}}$.  We also have
\[
\left\lvert b_{d_{\vec{w}}} \right\rvert r_0^{d_{\vec{w}}} = \left\lvert \psi(x_0) \right\rvert = t,
\]
so we can compute
\[
\left\lvert b_{d_{\vec{w}}} \right\rvert r^{d_{\vec{w}}} = \left\lvert b_{d_{\vec{w}}} \right\rvert r_0^{d_{\vec{w}}} \left( \frac{r}{r_0} \right)^{d_{\vec{w}}} = t \left( \frac{r}{r_0} \right)^{d_{\vec{w}}}
\]
and
\[
\left\lvert b_{d_{\vec{w}}} \right\rvert = \left\lvert b_{d_{\vec{w}}} \right\rvert r_0^{d_{\vec{w}}} r_0^{-d_{\vec{w}}} = t r_0^{-d_{\vec{w}}}.
\]

Now, suppose for the sake of a contradiction that $d_{\vec{w}} \neq 0$.

\textbf{Case 1:} $d_{\vec{w}} < 0$.  In this case, using the inequality $t > r/r_0$ from (2), we compute
\[
\left\lvert b_{d_{\vec{w}}} \right\rvert r^{d_{\vec{w}}} = t \left( \frac{r}{r_0} \right)^{d_{\vec{w}}} \geq t \left( \frac{r}{r_0} \right)^{-1} > 1
\]
so the inner Weierstrass degree of $\psi-1$ on $A_{\vec{w}}$ is $d_{\vec{w}}$.  However, using the inequality $t < 1$ from (2), we also have
\[
\left\lvert b_{d_{\vec{w}}} \right\rvert = t r_0^{-d_{\vec{w}}} < 1,
\]
which shows that the outer Weierstrass degree of $\psi-1$ on $A_{\vec{w}}$ is $0$.

\textbf{Case 2:} $d_{\vec{w}} > 0$.  In this case, using the inequality $t < 1$ from (2), we compute
\[
\left\lvert b_{d_{\vec{w}}} \right\rvert r^{d_{\vec{w}}} = t \left( \frac{r}{r_0} \right)^{d_{\vec{w}}} < 1,
\]
so the inner Weierstrass degree of $\psi-1$ on $A_{\vec{w}}$ is $0$.  However, using the inequality $t > r_0$ from (2), we also have
\[
\left\lvert b_{d_{\vec{w}}} \right\rvert = t r_0^{-d_{\vec{w}}} \geq t r_0^{-1} > 1,
\]
which shows that the outer Weierstrass degree of $\psi-1$ on $A_{\vec{w}}$ is $d_{\vec{w}}$.

In both cases, we find that $\psi-1$ has different inner and outer Weierstrass degrees on $A_{\vec{w}}$, and therefore has a zero on this annulus.  This contradicts that $\psi$ omits the value $1$ on $U$, as desired.
 
\textbf{Step 2: Prove the result for general $x \in U$.} As in the proof of (a), we need to show that $\psi$ has Weierstrass degree $0$ on each annulus $A_{\vec{v}}$.  Note that this is equivalent to showing that $\widetilde{\psi}$ is identically constant by Proposition \ref{prop:basicopen}(c).  Since a non-constant rational function must have both zeros and poles, it suffices to show that $\widetilde{\psi}$ has no poles.  Suppose for the sake of a contradiction that $\widetilde{\psi}$ has a pole.  This either means that there is a direction $\vec{v}$ not equal to the direction towards $\infty$ at $\zeta_{0,1}$ such that $\psi$ has negative Weirstrass degree on $A_{\vec{v}}$ or that $\psi$ has positive Weirstrass degree on the annulus towards $\infty$.  Again, we will assume that $\psi$ has negative Weierstrass degree on some $A_{\vec{v}}$; the case of $\infty$ can be reduced to this case by a change of coordinates.  As in Step 1, we must have that $A_{\vec{v}}$ is an annulus, not a disc, and therefore $\vec{v}$ is the direction towards some $a_j$.  Thus, $\psi\restrict{A_{\vec{v}}}$ has the form
\[
\psi(y) = \sum_{k = -\infty}^{\infty} c_k(y-a_j)^k.
\]
We also have $\left\lvert c_{d_{\vec{v}}} \right\rvert = \left\lvert \psi(\zeta_{0,1}) \right\rvert = t$ from Step 1.  Using the inequality $t > r_0 > r$ from (2) or the inequality $t > r$ from (3), we compute
\[
\left\lvert c_{d_{\vec{v}}} \right\rvert r^{d_{\vec{v}}} = t r^{d_{\vec{v}}} \geq t r^{-1} > 1,
\]
so the inner Weierstrass degree of $\psi-1$ on $A_{\vec{v}}$ is $d_{\vec{v}}$.  However, we also have $\left\lvert c_{d_{\vec{v}}} \right\rvert = t < 1$, which shows that the outer Weierstrass degree of $\psi-1$ on $A_{\vec{v}}$ is $0$.  Again, this contradicts that $\psi$ omits the value $1$ on $U$.
\end{proof}

Finally, we state a necessary and sufficient criterion for a net of type I points in $\mathbb{A}^{1,\text{an}}_{\mathbb{K}}$ to converge to a type II, III, or IV point.

\begin{prop} \label{prop:nets}
Let $\left\lbrace \gamma_j \right\rbrace_{j \in J}$ be a net of type I points in $\mathbb{A}^{1,\text{an}}_{\mathbb{K}}$ indexed by some directed set $(J, \leq)$.  Let $\xi \in \mathbb{A}^{1,\text{an}}_{\mathbb{K}}$ be a type II, III, or IV point.  Then, $\lim\limits_{j \in J} \gamma_j = \xi$ if and only if the following conditions hold
\begin{enumerate}[$(1)$]
\item for every finite tangent direction $\vec{w}$ at $\xi$,
\[
\lim\limits_{j \in J} \left\lvert \gamma_j - \gamma \right\rvert = \text{diam}(\xi)\]
for some (or equivalently any) $\gamma \in B(\vec{w})$
\item $\lim\limits_{j \in J} t_j = \text{diam}(\xi)$, where
\[
t_j := \min \left\lbrace \tau \in \mathbb{R}_{>0} \mid \xi \in \overline{D}(\gamma_j,\tau) \right\rbrace.
\]
\end{enumerate}
\end{prop}

\begin{proof}
This follows immediately from the fact that the basic open sets described in \S \ref{ssec:bpl} form a base for the topology of $\mathbb{A}^{1,\text{an}}_\mathbb{K}$.
\end{proof}

Note that condition (1) in Proposition \ref{prop:nets} holds for some finite direction $\vec{w}$ if condition (2) is satisfied and there exists $j_0 \in J$ such that $\gamma_j \not \in B(\vec{w})$ for all $j \geq j_0$.

\begin{proof}[Proof of the Non-archimedean $\lambda$-Lemma]
Since $U$ is covered by basic open sets and analyticity is a local property, we may assume without loss of generality that $U$ is itself a basic open set.  After applying a linear fractional transformation on $U$ and perhaps shrinking $U$ again, we may assume that $U$ has the form
\[
U = D(0,1/r) \setminus \left( \coprod_{i=1}^n \overline{D}(a_i,r) \right)
\]
with $r<1$ and each $a_i$ in a different tangent direction at the Gauss point $\zeta_{0,1}$.  After applying linear fractional transformations varying analytically with $\lambda$, we may assume that $0$ and $\infty$ are elements of $E$ and are kept fixed by the analytic motion.  

\textbf{Step 1: Extend the motion to type I points.}  Let $s \in \left\lvert \mathcal{H}(x_0)^{\times} \right\rvert$ be given.  Our first step is to show that each $\phi_\lambda$ is uniformly continuous on $E_s := \left\lbrace \alpha \in E \mid \left\lvert \alpha \right\rvert_{x_0} = s \right\rbrace$, using the absolute value on $\mathcal{H}(x_0)$ as the metric on $E$ and viewing the metric space $\mathcal{H}(\lambda) \subseteq \mathbb{P}^{1,\text{an}}_\lambda$ as the codomain.  For each $\alpha \in E$, denote $\psi_\alpha(\lambda) := (\pi_1 \circ \phi_\lambda)(\alpha).$  We also will show that the $\psi_\alpha$ for $\alpha \in E_s$ converge uniformly on compact subsets of $U$ to the $\psi_\gamma$ for type I points $\gamma$ in the closure of $E_s$.

We claim that $\left\lvert \psi_{\alpha}(\zeta_{0,1}) \right\rvert$ can take on at most two values for all $\alpha \in E_s$.  One possibility is that $\left\lvert \psi_{\alpha}(\zeta_{0,1}) \right\rvert = \left\lvert \psi_{\alpha}(x_0) \right\rvert = s$, which is only possible if $s \in \left\lvert K^{\times} \right\rvert$.  Choose $\sigma \in K$ with $\left\lvert \sigma \right\rvert = s$ in this case.  This is of course the only possibility if $x_0 = \zeta_{0,1}$.  Otherwise, any $\psi_\alpha$ with $\left\lvert \psi_\alpha \right\rvert$ non-constant must have the same Weierstrass degree on the annulus $A \subseteq U$ with one boundary point at $\zeta_{0,1}$ containing $x_0$ (otherwise the difference of two such $\psi_\alpha$ would have a zero on this annulus contradicting the injectivity of the $\phi_\lambda$).  Also, those $\psi_\alpha$ must all map $\zeta_{0,1}$ to the segment between $0$ and $\infty$ by Proposition \ref{prop:basicopen}(b).  Therefore, those $\psi_\alpha$ must map $\zeta_{0,1}$ to the same point on the this segment, say $\zeta_{0,s_1}$.  Choose $\sigma_1 \in K$ with $\left\lvert \sigma_1 \right\rvert = s_1$ if such $\psi_{\alpha}$ exist.  Thus, for each $\alpha \in E_s$, there exists $\sigma_{\alpha} \in \left\lbrace \sigma, \sigma_1 \right\rbrace$ such that $\left\lvert \psi_{\alpha}(\zeta_{0,1}) \right\rvert = \left\lvert \sigma_{\alpha} \right\rvert$ (where the subset $\left\lbrace \sigma, \sigma_1 \right\rbrace$ only has one element if $\sigma$ and $\sigma_1$ have not both been chosen).  The reduction $\widetilde{\psi_\alpha/\sigma_{\alpha}}$ is a rational function with coefficients in the residue field $\widetilde{K}$.  Moreover, injectivity of the $\phi_\lambda$ and the fact that $0$ and $\infty$ are kept constant by the analytic motion implies that the non-constant reductions and all the differences between non-constant reductions $\widetilde{\psi_\alpha/\sigma_{\alpha}}$ with the same $\sigma_{\alpha}$ value can only have zeros and poles in the set $\left\lbrace \infty \right\rbrace \cup \left\lbrace \widetilde{a_i} \mid 1 \leq i \leq n \right\rbrace$.  Applying Lemma $\ref{lem:sunitdiff}$ with  $L=\widetilde{K}(T)$ and $S = \left\lbrace (1/T) \right\rbrace \cup \left\lbrace (T-\widetilde{a_i})  \mid 1 \leq i \leq n \right\rbrace$, we conclude that
\[
\sup_{\alpha \in E_s, ~\sigma_{\alpha} = \sigma} \text{deg}(\widetilde{\psi_\alpha/\sigma}) < \infty \hspace{20 pt} \text{and} \hspace{20pt} \sup_{\alpha \in E_s, ~\sigma_{\alpha} = \sigma_1 } \text{deg}(\widetilde{\psi_\alpha/\sigma_1}) < \infty.
\]
Combining these degree bounds and using Proposition \ref{prop:basicopen}(d), we obtain that there exists $R \geq 0$ such that
\[
\bigcup_{\alpha \in E_s} \text{im}(\psi_\alpha) \subseteq \overline{D}(0,R).
\]
Note that this also implies that
\[
\bigcup_{\alpha,\beta \in E_s} \text{im}(\psi_\alpha-\psi_\beta) \subseteq \overline{D}(0,R).
\]

Now, let a compact subset $W$ of $U$ be given.  Without loss of generality, we may assume that $W$ has the form
\[
W = \overline{D}(0,1/u) \setminus \left( \coprod_{i=1}^n D(a_i,u) \right)
\]
with $u > r$.  Let $\epsilon > 0$ be given.  Choose $N$ sufficiently large that $(u/r)^N > R/\epsilon$.  Take $\delta = \epsilon u^{2N(n+1)}$, and let $\alpha, \beta \in E_s$ with $\left\lvert \alpha-\beta \right\rvert_{x_0} < \delta$ be given.  Note that by injectivity of the $\phi_\lambda$, we have that $\psi_{\alpha, \beta} := \psi_\alpha - \psi_\beta$ has equal inner and outer Weierstrass degrees on annuli contained in $U$ with one endpoint at $\zeta_{0,1}$.  Denote $s_2 := \left\lvert \psi_{\alpha,\beta}(\zeta_{0,1}) \right\rvert$ and choose $\sigma_2 \in K$ with $\left\lvert \sigma_2 \right\rvert = s_2$.  Then, these Weierstrass degrees are exactly the orders of vanishing of $\widetilde{\psi_{\alpha, \beta}/\sigma_2}$ at the corresponding points of $\widetilde{K}$ by Proposition \ref{prop:basicopen}(c).

\textbf{Case 1:} One of these Weierstrass degrees has absolute value larger than $N(n+1)$.  Then, the degree of $\widetilde{\psi_{\alpha, \beta}/\sigma_2}$ is also larger than $N(n+1)$.  Note that $\psi_{\alpha, \beta}$ has no poles, and therefore $\widetilde{\psi_{\alpha, \beta}/\sigma_2}$ can map only $\widetilde{\infty}$ and the $\widetilde{a_i}$ to $\widetilde{\infty}$.  Since the total number of poles counted with multiplicity is the degree of $\widetilde{\psi_{\alpha, \beta}/\sigma_2}$ and there are most $n+1$ poles, we obtain that the largest multiplicity of a pole is at least $N(n+1)/(n+1)=N$.  Assume the maximum is attained at some $\widetilde{a_j}$; the case of $\widetilde{\infty}$ is similar.  From above, $\text{im}(\psi_\alpha-\psi_\beta) \subseteq \overline{D}(0,R)$, which implies that $\psi_{\alpha, \beta}(\zeta_{a_j,r}) \in \overline{D}(0,R)$, using the continuous extension of $\psi_{\alpha,\beta}$ to $\overline{U}$ from Proposition \ref{prop:basicopen}(a).  Since $\psi_{\alpha, \beta}$ has Weierstrass degree $\leq -N$ on the open annulus between $\zeta_{a_j,u}$ and $\zeta_{a_j,r}$, we may use Propostion \ref{prop:scaling} to compute
\begin{align*}
d_\mathbb{H} \left( \psi_{\alpha, \beta}(\zeta_{a_j,r}), \psi_{\alpha, \beta}(\zeta_{a_j,u}) \right) & \geq N d_\mathbb{H} \left( \zeta_{a_j,r}, \zeta_{a_j,u} \right) = N (\log r -\log u) \\
&> \log(R/\epsilon),
\end{align*}
which implies that $\psi_{\alpha, \beta}(\zeta_{a_j,u}) \in \overline{D}(0,\epsilon)$.  Moreover, by our choice of $a_j$, we have that no points of $\psi_{\alpha, \beta}(W)$ are contained in the direction toward $\infty$ at $\psi_{\alpha, \beta}(\zeta_{a_j,u})$ (see the proof of Proposition \ref{prop:basicopen}(d)).  We conclude that 
\[
\left\lvert \phi_x(\alpha)-\phi_x(\beta) \right\rvert_x = \left\lvert (\psi_\alpha - \psi_\beta)(x) \right\rvert < \epsilon
\]
for all $x \in W$.

\textbf{Case 2:} All these Weierstrass degrees have absolute value less than $N(n+1)$.  If they are all $0$, then $\left\lvert \psi_{\alpha,\beta} \right\rvert$ is constant and we are finished.  Otherwise since the hyperbolic diameter of $\partial W$ is $\log(1/u^2)$, we have that the hyperbolic diameter of $\partial \psi_{\alpha,\beta}(W) \cap [0, \infty]$ is less than or equal to $N(n+1) \log(1/u^2)$, where $[0, \infty]$ is the segment from $0$ to $\infty$ in $\mathbb{P}^{1,\text{an}}$.  Also, using that
\[
\left\lvert \psi_{\alpha,\beta}(x_0) \right\rvert = \left\lvert \alpha-\beta \right\rvert_{x_0} < \delta,
\]
we obtain that there is a point $\xi_1 \in \partial \psi_{\alpha,\beta}(W) \cap [0, \infty]$ with $\text{diam}(\xi_1) < \delta$.  Combining these facts, we compute that the boundary point $\xi_2$ of $\psi_{\alpha,\beta}(W)$ closest to $\infty$ satisfies
\begin{align*}
\log \text{diam}(\xi_2) &\leq \log \text{diam}(\xi_1) + d_\mathbb{H}(\xi_1,\xi_2) < \log \delta+ N(n+1) \log(1/u^2) \\
&= \log \epsilon +2N(n+1) \log(u) - 2N(n+1) \log(u) \\
&= \log \epsilon.
\end{align*}
Again, we conclude that
\[
\left\lvert \phi_x(\alpha)-\phi_x(\beta) \right\rvert_x = \left\lvert (\psi_\alpha - \psi_\beta)(x) \right\rvert < \epsilon
\]
for all $x \in W$.

Thus, each $\phi_\lambda$ can be extended to a continuous function on $\overline{E_s} \cap \mathbb{P}^1(\mathcal{H}(x_0))$, and since the convergence is uniform on compact subsets of $U$, the extension varies analytically with $\lambda$.  Since $s \in \mathbb{R}_+$ was arbitrary, the same holds for $\overline{E} \cap \mathbb{P}^1(\mathcal{H}(x_0))$.

\textbf{Step 2: Extend the motion to non-type I points.}  Our next step is to extend the motion to non-type I points in $\overline{E}$.  Since $\overline{E}$ is not in general a metric space, we will need to use the language of nets for this.  Let $\eta \in \overline{E} \setminus \mathbb{P}^1(\mathcal{H}(x_0))$ be given, and let $\left\lbrace \alpha_j \right\rbrace_{j \in J}$ be a net of elements of $E$ indexed by some directed set $(J, \leq)$ that converges to $\eta$.  Let $x \in U$ be given.  The goal is to show that the net $\left\lbrace \phi_x (\alpha_j) \right\rbrace_{j \in J}$ converges in $\mathbb{A}^{1, \text{an}}_x$.  There are four cases.  In each case, we will use the notation $r_0 := \min_{1 \leq i \leq n} \left\lbrace \left\lvert T - a_i \right\vert_{x_0}, 1/ \left\lvert x_0 \right\rvert \right\rbrace$ and
\[
\rho :=
\begin{cases}
\max(r/r_0,r_0) &\text{ if } r_0<1 \\
r &\text{ if } r_0=1,
\end{cases}
\]
but otherwise notation introduced in one case will not carry over into the next.

\textbf{Case 1:} $\left\lbrace \alpha_j \right\rbrace$ approaches $\eta$ strictly from above.  By this, we mean either that $\eta$ is type IV or that no $\alpha_j$ is contained in the closed disc associated to a type II or III $\eta$.  In this case, denoting
\[
t_j := \min \left\lbrace \tau \in \mathbb{R}_{>0} \mid \eta \in \overline{D}(\alpha_j,\tau) \right\rbrace
\]
and $t := \text{diam}(\eta)$, we have $\lim_{j \in J} t_j = t$ by Proposition \ref{prop:nets} and $t_j > t$ for all $j$.  Because of this, for each $j \in J$, there exists $\ell_j \in J$ such that $\ell_j \geq j$ and $t_k < t_j$ for all $k \geq \ell_j$.  It follows that $\left\lvert \alpha_j-\alpha_k \right\rvert_{x_0} = t_j$ for all $j \in J$ and $k \geq \ell_j$.  Since $\lim_{j \in J} t_j = t$, we can also find $j_0 \in J$ such that $\rho t_j < t$ for all $j \geq j_0$.  Then, we have that $\left\lvert (\psi_{\alpha_j}-\psi_{\alpha_k})(\zeta_{0,1}) \right\rvert = \left\lvert (\psi_{\alpha_j}-\psi_{\alpha_{\ell}})(\zeta_{0,1}) \right\rvert$ for all $j \geq j_0$ and $k, \ell \geq \ell_j$, since otherwise $\psi_{\alpha_j}-\psi_{\alpha_k}$ and $\psi_{\alpha_j}-\psi_{\alpha_{\ell}}$ would have different Weierstrass degrees on the annulus in the direction of $x_0$ contained in $U$ which would imply that $\psi_{\alpha_k}-\psi_{\alpha_{\ell}}$ has a zero on this annulus.  Denote by $s_j$ this common value.  Now, let $j \geq j_0$, $k \geq \ell_j$, and $\ell \geq \ell_k$ be given.  Applying Lemma \ref{lem:injective} to the analytic functions $\psi_{\alpha_j}-\psi_{\alpha_{\ell}}$ and $\psi_{\alpha_k}-\psi_{\alpha_{\ell}}$, we obtain that $s_k/ s_j = t_k/t_j <1$, and hence
\[
\left(\left(\psi_{\alpha_j}-\psi_{\alpha_{\ell}} \right)/s_j\right)^{\sim} - \left(\left(\psi_{\alpha_j}-\psi_{\alpha_{k}} \right)/s_j\right)^{\sim} = \widetilde{s_k/s_j} \left(\left(\psi_{\alpha_k}-\psi_{\alpha_{\ell}} \right)/s_k\right)^{\sim} = 0.
\]
Therefore,
\[
\left(\frac{\psi_{\alpha_j}-\psi_{\alpha_{\ell}}}{\psi_{\alpha_j}-\psi_{\alpha_k}}\right)^{\sim} = 1
\]
and hence the quotient inside the parentheses has Weierstrass degree zero.  We conclude that
\[
\sigma_j := \left\lvert (\psi_{\alpha_j}-\psi_{\alpha_k})(x) \right\rvert = \left\lvert \phi_x(\alpha_j)-\phi_x(\alpha_k) \right\rvert_x
\]
is independent of $k \geq \ell_j$.  Lemma \ref{lem:injective} applied to the same analytic functions also shows that $\sigma_k/\sigma_j = t_k/t_j$ for all $j \geq j_0$ and $k \geq \ell_j$.  Thus, $\left\lbrace \sigma_j \right\rbrace_{j \in J}$ is Cauchy and $\sigma := \lim\limits_{j \in J} \sigma_j > 0$.  It follows from Proposition \ref{prop:nets} that the net $\left\lbrace \phi_x(\alpha_j) \right\rbrace_{j \in J}$ converges in $\mathbb{A}^{1,\text{an}}_x$ to a type II, III, or IV point with $\text{diam} = \sigma$.

\textbf{Case 2:} There is exactly one direction below $\eta$ containing elements of the net $\left\lbrace \alpha_j \right\rbrace_{j \in J}$.  By this, we mean that $\eta$ is type II or type III, there is some element $\alpha$ of the net contained in the associated closed disc, and $\left\lvert \alpha_j-\alpha \right\rvert_{x_0} \neq \text{diam}(\eta)$ for all $j \in J$.  Then, denoting $t := \text{diam}(\eta)$, we have $\eta = \zeta_{\alpha, t}$.  Denote $t_j := \left\lvert \alpha_j-\alpha \right\rvert$ for $j \in J$.  By Proposition \ref{prop:nets}, we have that $\lim\limits_{j \in J} t_j = t$.  Because of this, we can find $j_0 \in J$ such that
\[
\left\lvert \log(t_j)-\log(t) \right\rvert < \log(1/\rho)
\]
for all $j \geq j_0$, and for each $j \in J$, we can find $\ell_j \in J$ such that $\left\lvert t_k-t \right\rvert < \left\lvert t_j-t \right\rvert$ for all $k \geq \ell_j$.  Let $j \geq j_0$ and $k \geq \ell_j$ be given and apply Lemma \ref{lem:injective} to the analytic functions $\psi_{\alpha_j}-\psi_{\alpha}$ and $\psi_{\alpha_k}-\psi_{\alpha}$.  Denoting
\[
\sigma_j := \left\lvert (\psi_{\alpha_j}-\psi_{\alpha})(x) \right\rvert = \left\lvert \phi_x(\alpha_j)-\phi_x(\alpha) \right\rvert_x,
\]
we obtain that $\sigma_k/\sigma_j = t_k/t_j$ for all $j \geq j_0$ and $k \geq \ell_j$.  This shows that $\left\lbrace \sigma_j \right\rbrace_{j \in J}$ is Cauchy and $\sigma := \lim\limits_{j \in J} \sigma_j > 0$.  It follows from Proposition \ref{prop:nets} that the net $\left\lbrace \phi_x(\alpha_j) \right\rbrace_{j \in J}$ converges in $\mathbb{A}^{1,\text{an}}_x$ to $\zeta_{\phi_x(\alpha),\sigma}$.

\textbf{Case 3:} There are at least two directions below $\eta$ containing elements of the net $\left\lbrace \alpha_j \right\rbrace_{j \in J}$ and every direction at $\eta$ is eventually excluded.  By this, we mean that $\eta$ is type II or type III and for every direction $\vec{w}$ at $\eta$ there exists $j_0 \in J$ such that $\alpha_j \not\in B(\vec{w})$ for all $j \geq j_0$.  Choose $\alpha, \beta$ in the net that are in distinct directions at $\eta$ and are not in the direction towards $\infty$.  After applying a linear fractional transformation depending analytically on $\lambda$, we may assume that $\alpha = 0$, $\beta = 1$, $\infty \in E$, and $0,1, \infty$ are fixed by the motion.  Note that after this change of coordinates, we have $\eta = \zeta_{0,1} \in \mathbb{A}^{1,\text{an}}_{x_0}$.  We will continue to use the notation $\eta$ for this point to distinguish from the point $\zeta_{0,1} \in \mathbb{A}^{1,\text{an}}$.

Now, applying Lemma \ref{lem:sunitdiff} to the set $A := \left\lbrace \widetilde{\psi_{\alpha_j}} \mid j \in J \right\rbrace \setminus \left\lbrace 0, \infty \right\rbrace$ with $u_0 = 1$, we obtain that $A$ is either finite or consists only of constants in $\widetilde{K}^\times$.  

Suppose for the sake of a contradiction that for every $j_0 \in J$ there exists $j \geq J$ such that $\widetilde{\psi_{\alpha_j}}$ is non-constant.  Then, $A$ is finite, so by the pigeonhole principle there is a non-constant element $v \in A$ such that for every $j_0 \in J$ there exists $j \geq j_0$ with $\widetilde{\psi_{\alpha_j}} = v$.  Choose some $k \in J$ with $\widetilde{\psi_{\alpha_k}} = v$.  Then, for every $j_0 \in J$, there exists $j \geq j_0$ such that $\widetilde{\psi_{\alpha_j} - \psi_{\alpha_k}} = 0$ and hence such that
\[
\left\lvert \alpha_j - \alpha_k \right\rvert_{x_0} = \left\lvert (\psi_{\alpha_j}-\psi_{\alpha_k})(x_0) \right\rvert < 1 = \text{diam}(\eta).
\]
This contradicts the assumption that the net $\left\lbrace \alpha_j \right\rbrace$ eventually avoids the direction towards $\alpha_k$ at $\eta$.

Thus, after replacing $J$ with a subset of the form $\left\lbrace j \geq j_0 \right\rbrace$, we may assume that $\widetilde{\psi_{\alpha_j}}$ is constant for all $j \in J$.  It follows that $\psi_{\alpha_j}(U)$ is contained in a connected component of $\mathbb{A}^{1,\text{an}} \setminus \left\lbrace \zeta_{0,1} \right\rbrace$ for all $j \in J$.  Therefore, for all $j,k \in J$, we have that $\alpha_j$ and $\alpha_k$ are in the same connected component of $\mathbb{A}^{1,\text{an}}_{x_0} \setminus \left\lbrace \eta \right\rbrace$ if and only if $\phi_x(\alpha_j)$ and $\phi_x(\alpha_k)$ are in the same connected component of $\mathbb{A}^{1,\text{an}}_x \setminus \left\lbrace \zeta_{0,1} \right\rbrace$.  Since $\left\lbrace \alpha_j \right\rbrace_{j \in J}$ eventually avoids all directions at $\eta$, this means that $\left\lbrace \phi_x(\alpha_j) \right\rbrace_{j \in J}$ eventually avoids all directions at the point $\zeta_{0,1}$ in $\mathbb{A}^{1,\text{an}}_x$.  We conclude that $\left\lbrace \phi_x(\alpha_j) \right\rbrace_{j \in J}$ converges to the point $\zeta_{0,1}$ in $\mathbb{A}^{1,\text{an}}_x$ by the remark following Proposition \ref{prop:nets}.

\textbf{Case 4:} There are at least two directions below $\eta$ containing elements of the net $\left\lbrace \alpha_j \right\rbrace_{j \in J}$ and there is a direction at $\eta$ that is not eventually excluded.  By this, we mean that $\eta$ is type II or type III and there exist elements $\alpha$ and $\beta$ of the net that are contained in the closed disc associated to $\eta$ such that $\left\lvert \alpha-\beta \right\rvert_{x_0} = \text{diam}(\eta)$ and there is a direction at $\eta$ that does not satisfy the condition stated in Case 3.  After possibly choosing a new $\alpha$ and $\beta$, we may assume that $\alpha$ is in a direction that is not eventually excluded.  As in Case 3, apply a linear fractional transformation depending analytically on $\lambda$ so that $\alpha = 0$, $\beta = 1$, $\infty \in E$, and $0, 1, \infty$ are fixed by the motion.  Then, the condition that the direction towards $\alpha = 0$ is not eventually excluded can be stated as: for every $j_1 \in J$, there exists $j \geq j_1$ such that $\left\lvert \alpha_j \right\rvert_{x_0} < 1$.  Moreover, by Proposition $\ref{prop:nets}$, we have that $\lim\limits_{j \in J} \left\lvert \alpha_j \right\rvert_{x_0}= 1$.

Therefore, we can find $j_0 \in J$ such that
\[
\rho < \left\lvert \alpha_k \right\rvert_{x_0} < 1/\rho
\]
for all $k \geq j_0$ and such that
\[
\rho < \left\lvert \alpha_{j_0} \right\rvert_{x_0} < 1.
\]
Then, for any $k \geq j_0$ such that $\left\lvert \alpha_k \right\rvert_{x_0} \neq 1$, we can apply Lemma $\ref{lem:injective}$ to the analytic functions $\psi_{\alpha_k}$ and $1$, to obtain that $\left\lvert \psi_{\alpha_k}(\lambda) \right\rvert / \left\lvert \psi_{\alpha_k}(x_0) \right\rvert = 1$ for all $\lambda \in U$.  In particular, this applies to $\psi_{\alpha_{j_0}}$.  Using this, for any $\alpha_j$ with $\left\lvert \alpha_j \right\rvert_{x_0} = 1$, we can apply Lemma $\ref{lem:injective}$ to the analytic functions $\psi_{\alpha_j}$ and $\psi_{\alpha_{j_0}}$ to obtain that
\begin{equation} \label{eq:case4}
\frac{\left\lvert \psi_{\alpha_{j}}(\lambda) \right\rvert}{\left\lvert \psi_{\alpha_{j}}(x_0) \right\rvert} = \frac{\left\lvert \psi_{\alpha_{j_0}}(\lambda) \right\rvert}{\left\lvert \psi_{\alpha_{j_0}}(x_0) \right\rvert} = 1 \tag{*}
\end{equation}
for all $\lambda \in U$.  This means that $\psi_{\alpha_j}(U)$ is contained in a connected component of $\mathbb{A}^{1,\text{an}} \setminus \left\lbrace \zeta_{0,1} \right\rbrace$ for all $\alpha_j$ with $\left\lvert \alpha_j \right\rvert_{x_0} = 1$.  Above, we showed that this is also true for $\left\lvert \alpha_k \right\rvert_{x_0} \neq 1$ provided that $k \geq j_0$.  Therefore, after replacing $J$ by the subset $\left\lbrace j \in J \mid j \geq j_0 \right\rbrace$, we have that $\psi_{\alpha_j}(U)$ is contained in a connected component of $\mathbb{A}^{1,\text{an}} \setminus \left\lbrace \zeta_{0,1} \right\rbrace$ for all $j \in J$.  As in Case 3, it follows that for any $j,k \in J$,  $\alpha_j$ and $\alpha_k$ are in the same connected component of $\mathbb{A}^{1,\text{an}}_{x_0} \setminus \left\lbrace \eta \right\rbrace$ if and only if $\phi_x(\alpha_j)$ and $\phi_x(\alpha_k)$ are in the same connected component of $\mathbb{A}^{1,\text{an}}_x \setminus \left\lbrace \zeta_{0,1} \right\rbrace$.

Now, we will use Proposition $\ref{prop:nets}$ to show that $\left\lbrace \phi_x(\alpha_j) \right\rbrace_{j \in J}$ converges to the point $\zeta_{0,1}$ in $\mathbb{A}^{1,\text{an}}_x$.  By the above computation, we have
\[
\lim\limits_{j \in J} \left\lvert \phi_x(\alpha_j) \right\rvert_x = \lim\limits_{j \in J} \left\lvert \psi_{\alpha_j}(x) \right\rvert = \lim\limits_{j \in J} \left\lvert \psi_{\alpha_j}(x_0) \right\rvert = \lim\limits_{j \in J} \left\lvert \alpha_j \right\rvert_{x_0} = 1.
\]
This shows  that condition (1) for the direction towards $0$ and condition (2) of Proposition $\ref{prop:nets}$ are satisfied.  We can also apply the remark after Proposition $\ref{prop:nets}$ to conclude that condition (1) is satisfied for any direction at the point $\zeta_{0,1,} \in \mathbb{A}^{1,\text{an}}_x$ that contains no elements of the net $\left\lbrace \phi_x(\alpha_j) \right\rbrace_{j \in J}$.  Therefore, it remains to show that condition (1) is satisfied for any direction $\vec{w}$ at $\zeta_{0,1,} \in \mathbb{A}^{1,\text{an}}_x$ that is not the direction towards $\infty$ or the direction towards $0$ and such that there is some $j(\vec{w}) \in J$ with $\phi_x(\alpha_{j(\vec{w})}) \in B(\vec{w})$.  Note that by the above remark about connected components, $\alpha_{j(\vec{w})}$ is not in the direction towards $\infty$ or the direction towards $0$ at $\eta$, and hence $\left\lvert \alpha_{j(\vec{w})} \right\rvert_{x_0} = 1$.  Thus, Proposition \ref{prop:nets} applied to the net $\left\lbrace \alpha_j \right\rbrace_{j \in J}$ and the direction towards $\alpha_{j(\vec{w})}$ at $\eta$ yields that $\lim\limits_{j \in J} \left\lvert \alpha_j-\alpha_{j(\vec{w})} \right\rvert_{x_0} = 1$.  This means that there exists $k(\vec{w}) \in J$ such that
\[
\rho < \left\lvert \alpha_j - \alpha_{j(\vec{w})} \right\rvert_{x_0} < 1/\rho
\]
for all $j \geq k(\vec{w})$.  If $\left\lvert \alpha_j - \alpha_{j(\vec{w})} \right\rvert_{x_0} = 1$, then $\left\lvert \phi_x(\alpha_j)-\phi_x(\alpha_{j(\vec{w})}) \right\rvert_x = 1$ by the above remark about connected components.  If $\left\lvert \alpha_j - \alpha_{j(\vec{w})} \right\rvert_{x_0} \neq 1$ and $j \geq k(\vec{w})$, then Lemma \ref{lem:injective} applied to the analytic functions $\psi_{\alpha_{j(\vec{w})}}-\psi_{\alpha_j}$ and $\psi_{\alpha_{j(\vec{w})}}$ gives
\[
\frac{\left\lvert (\psi_{\alpha_j}-\psi_{\alpha_{j(\vec{w})}})(x) \right\rvert}{\left\lvert \psi_{\alpha_{j(\vec{w})}}(x) \right\rvert} = \frac{\left\lvert \alpha_j-\alpha_{j(\vec{w})} \right\rvert_{x_0}}{\left\lvert \alpha_{j(\vec{w})} \right\rvert_{x_0}} = \left\lvert \alpha_j-\alpha_{j(\vec{w})} \right\rvert_{x_0}.
\]
Also, by (\ref{eq:case4}) we have $\left\lvert \psi_{\alpha_{j(\vec{w})}}(x) \right\rvert = 1$.  Therefore,
\[
\left\lvert \phi_x(\alpha_j)-\phi_x(\alpha_{j(\vec{w})}) \right\rvert_x = \left\lvert (\psi_{\alpha_j}-\psi_{\alpha_{j(\vec{w})}})(x) \right\rvert = \left\lvert \alpha_j-\alpha_{j(\vec{w})} \right\rvert_{x_0}
\]
for all $j \geq k(\vec{w})$.  We conclude that $\lim\limits_{j \in J} \left\lvert \phi_x(\alpha_j)-\phi_x(\alpha_{j(\vec{w})}) \right\rvert_x = 1$ which gives condition (1) of Proposition \ref{prop:nets}.

\textbf{Step 3: Injectivity.}  The first two steps show that we can extend the motion and that the extensions of $\phi_\lambda$ are continuous.  Uniqueness follows immediately from the fact that $E$ is dense in $\overline{E}$.  It remains to show that the extensions are injective.  For this, we can apply the same argument to extend the inverse maps (this is why proving the result for arbitrary $x_0 \in U$ was important).  The extensions of the inverse maps must continue to satisfy $\phi_\lambda \circ \phi_\lambda^{-1} = \phi_\lambda^{-1} \circ \phi_\lambda = \text{id}$, since these compositions are continuous and $E$ is dense in $\overline{E}$.
\end{proof}

\subsection{Proof of Theorem \ref{thm:main}} \label{ssec:pfmain}
We recall the statement of Theorem \ref{thm:main}.
\mainthm*

\begin{proof}
(1) $\Rightarrow$ (2):  Suppose for the sake of a contradiction that the Julia sets $J_\lambda$ move analytically on $U$ and (2) is not satisfied.

\textbf{Case 1:} $f_x$ has an unstably indifferent periodic point $\xi$ for some $x \in U$.  Choose an affinoid neighborhood $V := \mathcal{M}(\mathscr{A})$ of $x$ in $U$, which we may take to be the closure of a basic open set.  Passing to an iterate of the family, we may assume that $\xi$ is a fixed point.  After a change of coordinates, we may assume $\xi \neq \infty_x$.  Then, $\xi$ lies on the Berkovich $1$-period curve $\mathcal{C} := \mathcal{C}_1 \left(\left\lbrace f_\lambda \right\rbrace_{\lambda \in V} \right) \subseteq \mathbb{P}^{1,\text{an}} \times V$.  Also, $\pi_2(\xi) = x$ and $f_\lambda'(\xi) = \zeta_{0,1}$.  Since the coefficients of $f_\lambda(z)$ are analytic functions on the closed affinoid $V$, they must have finite topological degree on $V$.  Using this, one can check that
\[
f_\lambda ' : \mathcal{C} \longrightarrow \mathbb{P}^{1,\text{an}}
\]
is a finite morphism between Berkovich curves.  Since $\mathbb{P}^{1,\text{an}}$ is irreducible, Theorem \ref{thm:openmap} then yields that $f_\lambda ' : \mathcal{C} \longrightarrow \mathbb{P}^{1,\text{an}}$ is an open map.  Likewise, $ \pi_2: \mathcal{C} \longrightarrow V$ is an open map.  Choose a connected open neighborhood $W_1$ of $x$ in $V$ such that $\xi$ is the only preimage under $\pi_2$ of $x$ in its connected component $\Omega$ of $\left( \pi_2 \restrict{\mathcal{C}} \right)^{-1} (W_1)$.  Denote $W := f_\lambda ' \left( \Omega \right)$, which is open by the above remark.  In particular, since $\zeta_{0,1} \in W$, we have that $W$ also contains points outside $\overline{D}(0,1)$, and therefore there exists $\xi_1 \in \Omega$ such that $ \left\lvert f_\lambda ' ( \xi_1 ) \right\rvert > 1 $ .  Denote $x_1 := \pi_2(\xi_1)$.  Note that $\xi_1$ is a repelling fixed point for $f_{x_1}$ and hence $\xi_1 \in J_{x_1}$.  Since the Julia sets move analytically, we have a family of homeomorphisms
\[
\phi_\lambda: J_{x_1} \longrightarrow J_\lambda
\]
with $ f_\lambda \circ \phi_\lambda = \phi_\lambda \circ f_{x_1} $ on $J_{x_1}$.  Moreover, defining
\begin{align*}
\psi: V& \longrightarrow \mathbb{P}^{ 1,\text {an} } \times V \\
\lambda & \longmapsto \phi_\lambda ( \xi_1 ), 
\end{align*}
we know that $ \psi $ is analytic section of $\pi_2: \mathbb{P}^{1,\text{an}} \times V \longrightarrow V$, so its image consists of fiberwise type I points by Proposition \ref{prop:ansec}.  Furthermore, since $\xi_1 \in \mathcal{C}$, i.e. $\xi_1$ is a fiberwise type I fixed point, we have that the image of $\psi$ is contained in $\mathcal{C}$ as well by the conjugacy condition in the definition of an analytic motion of Julia sets.  Indeed, we obtain that $\psi$ is an analytic section for $\pi_2: \mathcal{C} \longrightarrow V$, so $\psi(W_1) \subseteq \left( \pi_2 \restrict{\mathcal{C}} \right)^{-1} (W_1)$.  Since $W_1$ is connected and $\psi$ is continuous, this means that $\psi(W_1) \subseteq \Omega$.  Now, by construction $\xi$ is the unique preimage under $\pi_2$ of $x$ in $\Omega$, which implies that $\psi(x) = \xi$.  We observe that
\[
\xi = \psi(x) = \phi_x (\xi_1) \in \text{im}(\phi_x ) = J_x.
\]
However, $\xi$ is an indifferent type I periodic point for $f_x$ hence is not contained in $J_x$.  This is the desired contradiction.

\textbf{Case 2:} $f_y$ has a type I repelling point $\eta$ with multiplicity $m \geq 2$ for some $y \in U$.  By Lemma $\ref{lem:ppcol}$, we see that $\eta$ is a type I point that is not contained in $\mathbb{P}^{1}(\mathcal{H}(y))$.  However, $\eta$ is also in the image of an analytic section
\[
U \longrightarrow \mathbb{P}^{1,\text{an}} \times U
\]
of $\pi_2$ (obtained from the analytic motion of $\eta$ as in Case 1), contradicting Proposition \ref{prop:ansec}.
\end{proof}

For the reverse implication, our main tool will be the Non-archimedean $\lambda$-Lemma.

\begin{proof}[Proof of Theorem \ref{thm:main} continued]
(2) $\Rightarrow$ (1): The strategy is to first show that the type I repelling periodic points move analytically and then to employ the Non-archimedean $ \lambda $-Lemma to extend the motion to the entire Julia set.  For the first step, let $n \geq 1$ be given and consider the Berkovich $n$-period curve $\mathcal{C} := \mathcal{C}_n(\left\lbrace f_\lambda \right\rbrace)$.  From (2), the map $\pi_2: \mathcal{C} \longrightarrow U$ has multiplicity $1$ at each type I repelling point.  By continuity of $\left\lvert f'_\lambda \right\rvert$, the subset $W$ of type I repelling points on $\mathcal{C}$ is an open set.  Applying Lemma \ref{lem:monic} and Lemma \ref{lem:ift}, we see that $\pi_2: W \longrightarrow U$ is a local isomorphism of Berkovich spaces.  Moreover, because there are no unstably indifferent periodic points, the number of type I repelling periodic points of period dividing $n$ for $f_\lambda$ is constant as $\lambda$ varies in $U$.  This means that $\pi_2 \restrict{W}$ is a local homeomorphism from a Hausdorff space to a connected Hausdorff space whose fibers all have the same finite cardinality $N$, and hence it is a topological covering map.  Since $U$ is simply connected, this covering must be trivial.  Thus, we obtain isomorphisms of Berkovich spaces from $U$ to each of the $N$ connected components of $W$, inverting $\pi_2$.  This gives an analytic motion $\left\lbrace \phi_\lambda \right\rbrace$ of the type I repelling periodic points of period dividing $n$, parametrized by $(U,x_0)$.  Moreover, for each type I repelling point $\beta \in J_{x_0}$ of period dividing $n$, we have that $\lambda \mapsto (f_\lambda \circ \phi_\lambda)(\beta)$ and $\lambda \mapsto (\phi_\lambda \circ f_{x_0})(\beta)$ are analytic and hence continuous on $U$, and therefore each maps $U$ into a connected component of $W$.  Thus, since they are both sections of $\pi_2 \restrict{W}$ and agree at the point $\lambda = x_0$, they must agree on all of $U$.  Since this works for all $n$, we obtain the desired analytic motion of all type I repelling periodic points that also satisfies $f_\lambda \circ \phi_\lambda = \phi_\lambda \circ f_{x_0}$.

Since the type I repelling periodic points of $f_{x_0}$ are dense in $J_{x_0}$, applying the Non-archimedean $\lambda$-Lemma yields an analytic motion $\left\lbrace \phi_\lambda \right\rbrace$ of $J_{x_0}$, $\phi_\lambda: J_{x_0} \longrightarrow \mathbb{P}^{1,\text{an}}_{\lambda}$.  The additional condition $f_\lambda \circ \phi_\lambda = \phi_\lambda \circ f_{x_0}$ continues to hold by continuity.  This condition shows that $\phi_{\lambda}(J_{x_0})$ contains the backward orbit of any type I repelling periodic point of $f_{\lambda}$ and hence is equal to $J_{\lambda}$ by \cite{BR10}*{Corollary 10.58}.
\end{proof}

\section*{Acknowledgements}
The author would like to thank his advisor Joe Silverman for all his advice and Rob Benedetto, Laura DeMarco, and Ken Jacobs for helpful discussions.

\Addresses

\end{document}